\documentclass[11pt]{amsart}
\usepackage[english]{babel}
\usepackage{amsmath,amscd,amssymb,amsthm,amsxtra}
\usepackage[mathscr]{eucal}
\usepackage{microtype}
\usepackage{mathtools}
\usepackage{accents}
\usepackage{color,graphicx,esint}
\usepackage{epsfig,epstopdf}
\usepackage{tikz-cd}
\usetikzlibrary{matrix,arrows,decorations.pathmorphing}
\usepackage[margin=1.0in]{geometry}
\usepackage{enumerate}
\usepackage{xfrac}

\DeclareMathAlphabet{\mathpzc}{OT1}{pzc}{m}{it}

\newcommand{\R}{\mathbb{R}}

\renewcommand{\S}{\mathbb{S}}

\def\vep{\varepsilon}
\def\vth{\vartheta}
\def\vphi{\varphi}
\def\1{{\bf 1}}
\def\e{\mathrm{e}}

\def\Nhd{\mathcal{N}}
\def\D{\mathcal{D}_s}
\def\De{\mathcal{D}_s^{\varepsilon}}

\DeclareMathOperator{\diag}{diag}
\DeclareMathOperator{\diam}{diam}
\DeclareMathOperator{\dist}{dist}

\DeclareMathOperator{\Id}{Id}
\DeclareMathOperator{\Lip}{Lip}
\DeclareMathOperator{\LSC}{LSC}

\DeclareMathOperator{\PV}{P.V.}

\DeclareMathOperator{\SC}{SC}
\DeclareMathOperator{\trace}{tr}
\DeclareMathOperator{\USC}{USC}

\newtheorem{thm}{Theorem}[section]
\newtheorem{prop}[thm]{Proposition}

\newtheorem{lem}[thm]{Lemma}

\theoremstyle{definition}
\newtheorem{defn}[thm]{Definition}
\newtheorem{rem}[thm]{Remark}

\numberwithin{equation}{section}

\allowdisplaybreaks

\author[Y. Jhaveri]{Yash Jhaveri}

\address{School of Mathematics\\
Institute for Advanced Study\\
1 Einstein Drive\\
Princeton, NJ \\
08540 USA}
\email{yjhaveri@ias.edu}

\author[P. R. Stinga]{Pablo Ra\'ul Stinga}

\address{Department of Mathematics\\
Iowa State University\\
396 Carver Hall, Ames\\
IA 50011, USA}
\email{stinga@iastate.edu}

\thanks{The first author was partially supported by ERC grant ``Regularity and Stability in Partial Differential Equations
(RSPDE)''.
The second author was supported by Simons Foundation grant 580911
and by grant MTM2015-66157-C2-1-P, MINECO/FEDER, UE, from Government of Spain}



\begin{document}

\title[The obstacle problem for a fractional Monge--Amp\`ere operator]{The obstacle problem for a \\ fractional Monge--Amp\`ere equation}

\begin{abstract}
We study the obstacle problem for a nonlocal, degenerate elliptic Monge--Amp\`ere equation.
We show existence and regularity of a unique classical solution to the problem and regularity of the free boundary.
\end{abstract}

\maketitle

\section{Introduction}

Obstacle problems for nonlocal operators appear in
optimal control, mathematical finance, biology, and elasticity, among other applied sciences.
The regularity of solutions and free boundaries for this type of nonlinear problem for the fractional Laplacian was studied by Silvestre in \cite{Silvestre-CPAM},
and by Caffarelli--Salsa--Silvestre in \cite{Caffarelli-Salsa-Silvestre},
and for homogeneous, translation invariant, purely nonlocal uniformly elliptic operators by
Caffarelli--Ros-Oton--Serra in \cite{Caffarelli-Ros-Oton-Serra}.

In this paper, we investigate the following nonlocal obstacle problem:
\begin{equation}
\label{eq:obstacle problem}
\begin{cases}
\D u\geq u-\phi&\hbox{in}~\R^n\\
u\leq\psi&\hbox{in}~\R^n\\
\D u=u-\phi&\hbox{in}~\{u<\psi\}\\
\displaystyle\lim_{|x|\to\infty}(u-\phi)(x)=0.
\end{cases}
\end{equation}
Here the fractional Monge--Amp\`ere operator $\D$ is defined for $s\in (1/2,1)$ and $u:\R^n\to\R$, $n\geq1$, as
\begin{equation}
\label{eq:MAs}
\D u(x)=\inf_{A\in\mathpzc{M}}\frac{c_{n,s}}{2} \int_{\R^n}\frac{u(x+y)+u(x-y)-2u(x)}{|A^{-1}y|^{n+2s}}\,dy
\end{equation}
where $x\in\R^n$, $c_{n,s}>0$, and $\mathpzc{M}$ is
the class of all positive definite symmetric matrices $A$ of size $n\times n$ such that $\det A=1$. 
In addition, $\phi$ is a function in $C^{2,\sigma}(\R^n)$, $\sigma > 0$, that is strictly convex in compact sets
and asymptotically linear at infinity, and conditions on the obstacle $\psi$ are given below.

The nonlocal operator \eqref{eq:MAs} was first introduced by L.~A.~Caffarelli and F.~Charro in \cite{Caffarelli-Charro}
as a fractional analogue to the classical Monge--Amp\`ere operator. In fact, if $u$ is a convex $C^2$ function,
then it can be checked that
\begin{equation}\label{local}
n(\det D^2u(x))^{1/n}=\inf_{A\in\mathpzc{M}}\trace (A^2D^2u)(x).
\end{equation}
If, in addition, $u$ is asymptotically linear at infinity, then
\[
\lim_{s\to1}\D u(x)=n(\det D^2u(x))^{1/n}
\]
(see \cite{Caffarelli-Charro}).
Like its local counterpart \eqref{local}, the fractional operator \eqref{eq:MAs} is degenerate elliptic.
Indeed, matrices of the form $A=\diag(\vep,1/\vep)$, $\vep>0$, in dimension 2,
are in $\mathpzc{M}$, and they degenerate as $\vep\searrow0$.
Thus, the existence and regularity theory for nonlocal elliptic equations previously developed
in \cite{Caffarelli-Silvestre-ARMA, Caffarelli-Silvestre-CPAM, Caffarelli-Silvestre-Annals},
see also \cite{Kriventsov, Ros-Oton-Serra-Duke, Serra},
does not directly apply to equations involving \eqref{eq:MAs}.
Caffarelli and Charro considered the problem
\begin{equation}
\label{eq:v}
\begin{cases}
\D \bar{u}(x)=\bar{u}(x)-\phi(x)&\hbox{in}~\R^n\\
\displaystyle\lim_{|x|\to\infty}(\bar{u}-\phi)(x)=0
\end{cases}
\end{equation}
where $\phi$ is as above, and they showed 
existence of a unique, globally Lipschitz and semiconcave classical solution $\bar{u}$.
In addition, it was proved in \cite{Caffarelli-Charro} that $\bar{u}$ has the crucial property that $\bar{u}>\phi$ in $\R^n$.
This ultimately implies that $\D$, when acting on $u$, is locally uniformly elliptic
and, consequently, $\nabla\bar{u}$ is locally H\"older continuous.

In view of the comparison principle for \eqref{eq:v} and in
order to have a meaningful obstacle problem, we assume that the obstacle $\psi\in C^{2,1}(\R^n)$ is such that
\begin{equation*}
\psi>\phi  \quad\hbox{in}~\R^n \quad\text{and}\quad \psi\leq \bar{u} \quad\hbox{in some compact set}~\mathcal{K}.
\end{equation*}
Here and in the remainder of this work, $\bar{u}$ denotes the solution to $\eqref{eq:v}$.
 
While the main tool in \cite{Guillen-Schwab} is the solution of a purely nonlocal and degenerate
elliptic obstacle problem, which stands as a replacement for the convex envelope in the fractional setting
(with the usual convex envelope being such an obstacle solution in the second order setting), we stress that the problem \eqref{eq:obstacle problem} is different in nature than the one in \cite{Guillen-Schwab}.
Obstacle problems for the local Monge--Amp\`ere equation \eqref{local}
were considered by Lee \cite{Lee} and Savin \cite{Savin}.
Our problem \eqref{eq:obstacle problem} can be seen as a parallel to \cite{Lee}.

Our first result establishes the existence and global regularity of a unique classical solution to \eqref{eq:obstacle problem}.
For the necessary notation, see section \ref{Section:Preliminaries}.

\begin{thm}
\label{thm: solution of obstacle problem}
There exists a unique classical solution $u$ to the obstacle problem \eqref{eq:obstacle problem}.
Moreover, $u$ is globally Lipschitz continuous and semiconcave with constants no larger than
\begin{equation}\label{Ms}
M_1 = \max \big\{[\phi]_{\Lip(\R^n)}, [\psi]_{\Lip(\R^n)}\big\} \quad\hbox{and}\quad M_2 = \max\big\{\SC(\phi), \SC(\psi) \big\},
\end{equation}
respectively, and the contact set $\{u=\psi\} \subset \mathcal{K}$ is compact. Furthermore,
\begin{equation}\label{eq:u bigger than phi}
u>\phi\quad\hbox{in}~\R^n.
\end{equation}
\end{thm}

The degenerate ellipticity of the fractional Monge--Amp\`ere operator \eqref{eq:MAs} prevents us from applying standard
techniques used to prove existence and uniqueness for nonlocal uniformly elliptic obstacle problems
\cite{Caffarelli-Ros-Oton-Serra, Silvestre-CPAM}. Therefore, to construct the solution $u$ to \eqref{eq:obstacle problem},
we need to devise a new strategy.
This is one of the main contributions of this paper.
To prove Theorem \ref{thm: solution of obstacle problem},
we consider a family of obstacle problems of the form \eqref{eq:obstacle problem}, but where $\D$ is replaced
by a truncated operator $\D^\vep$ (see section \ref{Section:Preliminaries}).
We build the solutions $u_\vep$ to such
uniformly elliptic nonlocal problems as the largest subsolution sitting below $\psi$ (see Theorem \ref{thm: eps-obstacle problem}), following the standard method.
More importantly, the key feature of the family of solutions $\{u_\vep\}_{\vep>0}$
is that it is uniformly globally Lipschitz continuous and semiconcave with constants no larger than
$M_1$ and $M_2$ (see \eqref{Ms}), respectively.
At this point, though, the degenerate ellipticity of $\D$ prevents us from applying the stability of viscosity solutions
under local uniform convergence.
Our crucial, delicate idea, which allows us to overcome this difficulty,
is to show that $u=\inf_{\vep>0}u_\vep$ remains strictly above $\phi$
(see Lemma \ref{lem:uzero separated}). See section \ref{Section:Existence} for details.

Next, we prove local H\"older estimates on $\nabla u$ outside of the contact set $\{u =\psi\}$
and across the free boundary $\partial\{u<\psi\}$.

\begin{thm}
\label{thm:regularity}
Let $u$ be the solution to the obstacle problem \eqref{eq:obstacle problem}.
\begin{enumerate}[(1)]
\item Let $\mathcal{O}$ be an open set and let $\mathcal{O}_\delta$, $\delta>0$, be a $\delta$-neighborhood of $\mathcal{O}$ such that
$\mathcal{O}_\delta\subset\subset\{u<\psi\}$.
There exists $\beta = \beta(n,s,\inf_{\mathcal{O}_\delta}(u -\phi),M_1,M_2) \in(0,1)$ such that $u \in C^{1,2s+\beta-1}(\mathcal{O})$ and
\[
\|u\|_{C^{1,2s+\beta-1}(\mathcal{O})}\leq C(1+\|u-\phi\|_{L^\infty(\R^n)}),
\]
where $C=C(\beta,\diam(\mathcal{O}))>0$.
\item Let $\mathcal{B}$ be a ball centered at the origin such that $\{u=\psi\}\subset\mathcal{B}$.
There exists $\tau=\tau(n,s,\inf_{4\mathcal{B}}(u -\phi),M_1,M_2) \in (0,1)$ such that
$u \in C^{1,\tau}(\mathcal{B})$ and
\[
\|u\|_{C^{1,\tau}(\mathcal{B})}\leq C(1+\|\psi\|_{C^{1,\tau}(\mathcal{B})}),
\]
where $C=C(\tau,\diam(\mathcal{B}))>0$.
\end{enumerate}
\end{thm}

The separation property \eqref{eq:u bigger than phi} and the global regularity of $u$
we found in Theorem \ref{thm: solution of obstacle problem}
will permit us to prove that, if we fix any ball $\mathcal{B}$, then $u$ solves
\begin{equation}\label{eq:B}
\begin{cases}
\D^\lambda u\geq u-\phi&\hbox{in}~\R^n\\
u\leq\psi&\hbox{in}~\R^n\\
\D^\lambda u=u-\phi&\hbox{in}~\{u<\psi\}\cap\mathcal{B}\\
\displaystyle\lim_{|x|\to\infty}(u-\phi)(x)=0
\end{cases}
\end{equation}
where $\D^\lambda$ is the truncated version of \eqref{eq:MAs},
with ellipticity constants depending on the gap between $u$ and $\phi$ in $\mathcal{B}$.
Then Theorem \ref{thm:regularity} will follow from local regularity estimates for uniformly elliptic nonlocal equations.
These ideas also demonstrate another important point of divergence between
our obstacle problem and uniformly elliptic nonlocal obstacle problems.
In \cite{Caffarelli-Ros-Oton-Serra}, solutions are shown to be $C^{1,\tau}(\R^n)$.
In contrast, since $\lim_{|x|\to\infty}(u-\phi)(x)=0$,
we cannot guarantee that $\D$, when acting on $u$, will be globally uniformly elliptic.
In particular, the H\"older exponents $\beta$ and $\tau$ in Theorem~\ref{thm:regularity}
degenerate as $\mathcal{O}$ drifts to infinity and $\mathcal{B}$ increases in size, respectively.

To study the regularity of the free boundary $\partial\{u<\psi\}$ and the behavior of $u$ near free boundary points,
we fix a ball $\mathcal{B}$ centered at the origin such that $\{u=\psi\}\subset\mathcal{B}$.
Then, $u$ satisfies the obstacle problem \eqref{eq:B}. Let 
\begin{equation}\label{eq:psi minus u}
v=\psi-u.
\end{equation}
Let $x_0\in\partial\{u<\psi\}=\partial\{v>0\}$ be a regular free boundary point (see Definition \ref{defn:regular}).
As in \cite{Caffarelli-Ros-Oton-Serra}, for $r>0$ and $\alpha\in(0,1)$ sufficiently small, we define the rescalings
\[
v_r(x)=\frac{v(x_0+rx)}{r^{1+s+\alpha}\theta(x_0,r)}\quad \hbox{for}~x\in\R^n
\]
where
\[
\theta(x_0,r)=\sup_{\rho\geq r}\frac{\|\nabla v(x_0+\cdot)\|_{L^\infty(B_\rho)}}{\rho^{s+\alpha}}.
\]

\begin{thm}\label{thm:blow up}
There exist a sequence $r_k\searrow0$, $1/4\leq K_0\leq 1$, and $\e_0\in\mathbb{S}^{n-1}$ such that
\[
v_{r_k}(x)\to K_0(\e_0\cdot x)^{1+s}_+\quad\hbox{in}~C^1_{\mathrm{loc}}(\R^n),~\hbox{as}~k\to\infty.
\]
\end{thm}

Observe that, unlike in \cite{Caffarelli-Ros-Oton-Serra}, we do not know whether or not $\D^\lambda u=u-\phi$ in the part of the noncoincidence
set $\{u<\psi\}$ that lies outside of $\mathcal{B}$. Hence, even after subtracting the obstacle from $u$ in \eqref{eq:B},
we cannot use known regularity results for the free boundary.
In fact, it is well known that the behavior at infinity of solutions to nonlocal equations can have dramatic consequences
on their local properties (see, for instance, \cite{Dipierro-Savin-Valdinoci}).
Moreover, our H\"older estimates for $\nabla u$ degenerate at infinity.
Nonetheless, we are able to overcome these issues due to the global regularity of $u$ we proved in Theorem \ref{thm: solution of obstacle problem}.
Indeed, this gives us enough control at infinity to be able to show Theorem \ref{thm:blow up}.
Finally, the separation property \eqref{eq:u bigger than phi}, the global regularity of $u$, and Theorem \ref{thm:blow up}
permit us to use the methods of \cite{Caffarelli-Ros-Oton-Serra}, in our setting, to 
obtain regularity of the free boundary.

\begin{thm}
\label{thm:2}
Let $u$ be the solution to \eqref{eq:obstacle problem}.
Let $\mathcal{B}$ be a ball centered at the origin such that $\{u=\psi\}\subset\mathcal{B}$.
There exists $\bar{\alpha} = \bar{\alpha}(n,s,\inf_{4\mathcal{B}}(u -\phi),M_1,M_2) \in(0,1)$ such that the following holds:
for any $\gamma \in (0,\bar{\alpha})$ and $\alpha \in (0,\bar{\alpha})$ such that $1+ s + \alpha < 2$ and for any $x_0 \in \partial \{ u < \psi \}$,
\begin{enumerate}[(1)]
\item either 
\[
\liminf_{r \searrow 0} \frac{|\{u = \psi\} \cap B_r(x_0)|}{|B_r(x_0)|} > 0 \quad\text{and}
\]
\[
\psi(x) - u(x) = cd^{1+s}(x) + o\big{(}|x-x_0|^{1+s+\alpha}\big{)}
\]
\item or
\[
\liminf_{r \searrow 0} \frac{|\{u = \psi\} \cap B_r(x_0)|}{|B_r(x_0)|} = 0 \quad\text{and}
\]
\[
\psi(x) - u(x) = o\big{(}|x-x_0|^{\min\{2s + \gamma,1+s+\alpha\}}\big{)}
\]
\item or
\[
\liminf_{r \searrow 0} \frac{|\{u = \psi\} \cap B_r(x_0)|}{|B_r(x_0)|} > 0 \quad\text{and}
\]
\[
\psi(x) - u(x) = o\big{(}|x-x_0|^{1+s+\alpha}\big{)}
\]
\end{enumerate}
where $d(x) = \dist(x, \partial\{u<\psi\})$ and $c > 0$.
Moreover, the set of points $x_0$ satisfying (1) is an open subset of the free boundary of class $C^{1,\gamma}$.
\end{thm}

The paper is organized as follows. In section \ref{Section:Preliminaries}, we establish some preliminary
results that will be needed for the rest of the work. The proofs of Theorems \ref{thm: solution of obstacle problem},
\ref{thm:regularity}, \ref{thm:blow up}, and \ref{thm:2} are presented in sections
\ref{Section:Existence}, \ref{sec: higher reg}, \ref{Section:blow up}, and \ref{Section:fb}, respectively.
\section{Preliminaries}
\label{Section:Preliminaries}

In this section, we recall some facts about
the fractional Monge--Amp\`ere operator $\D$, problem \eqref{eq:v}, and uniformly elliptic nonlocal operators.

\subsection{Notation}

Let $\mathcal{O}$ be an open subset of $\R^n$ and  $f : \mathcal{O} \to \R$.
We denote the Lipschitz constant of $f$ in $\mathcal{O}$ by
\[
[f]_{\Lip(\mathcal{O})}=\sup_{x,y\in\mathcal{O},\,x\neq y}\frac{|f(x)-f(y)|}{|x-y|}.
\]
For the second order incremental quotient of $f$ at $x$ in the direction of $y$, we write
\[
\delta(f,x,y) = f(x+y)+f(x-y)-2f(x).
\]
When $\mathcal{O}=\R^n$, we say that $f$ is semiconcave
 if there exists a constant $C>0$ such that $\delta(f,x,y)\leq C|y|^2$ for all $x,y\in\R^n$.
In this case,
\[
\SC(f)=\sup_{x,y\in\R^n}\frac{\delta(f,x,y)}{|y|^2}
\]
is the semiconcavity constant of $f$. 
Alternatively, $f$ is semiconcave if and only if $f(x)-C|x|^2/2$ is concave.

Let $\USC(\mathcal{O})$ (resp. $\LSC(\mathcal{O})$) be
the set of functions that are upper (resp. lower) semicontinuous
in $\mathcal{O}$. Define
\[
f^*(x)=\lim_{r\to0}\sup\big\{f(y):y\in\mathcal{O}~\hbox{and}~|y-x|<r\big\}\quad\hbox{for every}~x\in\mathcal{O}.
\]
We call $f^*$ the upper semicontinuous envelope of $f$ in $\mathcal{O}$; it is
the smallest $g\in\USC(\mathcal{O})$
satisfying $f\leq g$.

\begin{rem}\label{rmk: property 1 sc envelope}
A simple, useful property of the upper semicontinuous
envelope $f^*$  of $f$ in $\mathcal{O}$ is that for any $x_0\in\mathcal{O}$,
there exist points $y_k\in\mathcal{O}$ such that
$y_k\to x_0$ and $f(y_k)\to f^\ast(x_0)$, as $k\to\infty$.
(Note, we allow $y_k = x_0$ for all $k$.)
\end{rem}
 
\subsection{The fractional Monge--Amp\`ere operator}

We begin this subsection by providing some novel insight on the definition
of the fractional Monge--Amp\`ere operator $\D u$ in \eqref{eq:MAs}, which may be of independent interest.
Next, we precisely describe $\phi$. Then, we discuss the definition of viscosity solutions and some further properties of $\D u$ and the problem \eqref{eq:v}.

Recall that
\[
\mathpzc{M} =\big{\{}\hbox{symmetric positive definite matrices $A$ of size  $n \times n$ such that $\det A = 1$}\big{\}}.
\]
For any $A\in\mathpzc{M}$, we define the
constant coefficient second order elliptic operator
\[
L_Aw(x)=-\Delta[w\circ A](A^{-1}x)=-\trace (A^2D^2w)(x),
\]
see \eqref{local}.
Then, $L_A$ is nothing but a linear transformation of the Laplacian $-\Delta$.
For $s\in(0,1)$, consider the fractional power operator
\[
L_A^s=-(L_A)^s \quad\hbox{in}~\R^n.
\]

\begin{lem}
Let $w:\R^n\to\R$ such that
$$\int_{\R^n}\frac{|w(x)|}{(1+|x|)^{n+2s}}\,dx<\infty.$$
Let $\mathcal{O}$ be an open set. If $w\in C^{2s+\delta}(\mathcal{O})$ or $w\in C^{1,2s+\delta-1}(\mathcal{O})$ when $s\geq1/2$,
for some $\delta>0$, then, for any $x\in\mathcal{O}$,
\begin{equation}\label{eq:LAS}
\begin{aligned}
L_A^s w(x) &= c_{n,s}\PV \int_{\R^n}\frac{w(y)-w(x)}{|A^{-1}(y-x)|^{n+2s}}\,dy \\
&= \frac{c_{n,s}}{2} \int_{\R^n}\frac{w(x+y)+w(x-y)-2w(x)}{|A^{-1}y|^{n+2s}}\,dy \\
&= -(-\Delta)^s[w\circ A](A^{-1}x),
\end{aligned}
\end{equation}
where $c_{n,s}=\frac{4^s\Gamma(n/2+s)}{\pi^{n/2}|\Gamma(-s)|}>0$. As a consequence of \eqref{eq:LAS}, we have
\[
\D w(x)=\inf\big{\{}L_A^sw(x):A\in\mathpzc{M}\big{\}}.
\]
\end{lem}

\begin{proof}
The idea is to first prove \eqref{eq:LAS} for $w$ in the Schwartz class $\mathcal{S}$, by applying the method of semigroups
as in \cite[Lemma~5.1]{Stinga-Torrea}.
Then, for $w$ as in the hypotheses, one can use an approximation device exactly as done in \cite[Proposition~2.4]{Silvestre-CPAM}. We just sketch the steps here.
For $A\in\mathpzc{M}$ and $w\in\mathcal{S}$, the heat semigroup generated by
$L_A$ acting on $w$ is given explicitly by
\[
e^{tL_A}w(x)=\int_{\R^n}\frac{e^{-|A^{-1}x|^2/(4t)}}{(4\pi t)^{n/2}}\,w(x-y)\,dy,
\]
for $x\in\R^n$ and $t>0$. Then, since $e^{tL_A}1(x)=1$,
by Fubini's theorem (see  \cite[Lemma~5.1]{Stinga-Torrea}) and the change of variables $r=|A^{-1}x|^2/(4t)$,
\begin{align*}
L_A^s w(x) &= -(L_A)^sw(x) \\
&=\frac{1}{|\Gamma(-s)|}\int_0^\infty\big(e^{tL_A}w(x)-w(x)\big)\,\frac{dt}{t^{1+s}} \\
&= \frac{1}{|\Gamma(-s)|}\int_0^\infty\int_{\R^n}\frac{e^{-|A^{-1}x|^2/(4t)}}{(4\pi t)^{n/2}}(w(x-y)-w(x))\,dy\,\frac{dt}{t^{1+s}} \\
&= \PV\int_{\R^n}(w(x-y)-w(x))\bigg[\int_0^\infty\frac{e^{-|A^{-1}x|^2/(4t)}}{|\Gamma(-s)|(4\pi t)^{n/2}}\,\frac{dt}{t^{1+s}}\bigg]\,dy \\
&= c_{n,s}\PV \int_{\R^n}\frac{w(x-y)-w(x)}{|A^{-1}y|^{n+2s}}\,dy,
\end{align*}
as desired. The second identity in \eqref{eq:LAS} follows immediately from the first one, and the third one is
deduced via a simple change of variables.
\end{proof}

Now, we give the precise description of the function $\phi$ appearing in \eqref{eq:v} and \eqref{eq:obstacle problem}.
Let $\Gamma$ be a cone and $\eta : \R^n \to \R$ be such that
\[
|\eta(x)|\leq a|x|^{-\epsilon},\quad|\nabla\eta(x)|\leq a|x|^{-(1+\epsilon)},\quad\hbox{and}\quad|D^2\eta(x)|\leq a|x|^{-(2+\epsilon)}
\]
for some constants $a>0$ and $\epsilon \in (0,n)$. We let $\phi \in C^{2,\sigma}(\R^n)$, for some $\sigma>0$, be such that
\[
\phi(0)=0, \quad \nabla\phi(0)=0,\quad\hbox{and}\quad \phi=\Gamma+\eta \quad\hbox{near infinity}.
\]

We will work with viscosity solutions as defined in \cite[Definition~2.1]{Caffarelli-Charro}.

\begin{defn}
\label{defn:viscosity solution}
Let $\mathcal{O}$ be an open subset of $\R^n$. 
A function $w:\R^n\to\R$ such that $w \in \USC(\overline{\mathcal{O}})$ (resp. $w \in \LSC(\overline{\mathcal{O}})$)
is called a viscosity subsolution (resp. supersolution) to $\D w = w - \phi$ in $\mathcal{O}$, which we denote by
\[
\D w \geq w - \phi\quad\big(\hbox{resp.}~\D w \leq w - \phi\big)\quad\hbox{in}~\mathcal{O},
\]
if whenever
\begin{itemize}
\item[-] $x_0$ is a point in $\mathcal{O}$;
\item[-] $\mathcal{N} \subset \mathcal{O}$ is an open neighborhood of $x_0$;
\item[-] $P$ is a $C^2$ function on $\overline{\mathcal{N}}$;
\item[-] $P(x_0) = w(x_0)$; and
\item[-] $P(x) > w(x)$ (resp. $P(x) < w(x)$) for every $x \in \mathcal{N} \setminus \{ x_0 \}$;
\end{itemize}
then
\[
\D\vth(x_0) \geq \vth(x_0) - \phi(x_0) \quad(\text{resp. }\D\vth(x_0) \leq \vth(x_0) - \phi(x_0))
\]
where $\vth$ is defined as
\begin{equation}
\label{eq:vth}
\vth(x)=\begin{cases}
P(x)&\hbox{if}~x\in \mathcal{N}\\
w(x)&\hbox{if}~x\in\R^n\setminus \mathcal{N}.
\end{cases}
\end{equation}
When all of the items listed above are satisfied for some triplet $(P,x_0,\mathcal{N})$,
we say that $P$ is a $C^2$ function touching $w$ from above (resp. below) at $x_0$ in $\mathcal{N}$.
A viscosity solution $w$ is both a viscosity subsolution and a viscosity supersolution.
In particular, solutions are continuous by definition.
\end{defn}

From now on, any reference to a subsolution, supersolution, or solution will be in the viscosity sense.

Note that a semiconcave function can always be touched from above
by a quadratic polynomial at any point.

\begin{rem}
\label{rmk: viscosity nhd}
Let $P$ be a $C^2$ function touching $w$ from above (resp. below) at $x_0$ in $\Nhd$.
If $\Nhd'$ is any open subset of $\Nhd$ containing $x_0$,
then $P$ is a $C^2$ function that touches $w$ from above (resp. below) at $x_0$ in $\Nhd'$.
Define $\vth$ as in \eqref{eq:vth} and let
\[
\vth'(x)=\begin{cases}
P(x)&\hbox{if}~x\in\Nhd'\\
w(x)&\hbox{if}~x\in\R^n \setminus\Nhd'.
\end{cases}
\]
Then, $\vth\geq\vth'$ (resp. $\vth\leq\vth'$) in $\R^n$ and $\vth(x_0)=\vth'(x_0)$,
so that $\delta(\vth,x_0,y)\geq\delta(\vth',x_0,y)$ for every $y\in\R^n$
(resp. $\delta(\vth,x_0,y) \leq \delta(\vth',x_0,y)$).
It follows that $L_A^s\vth(x_0) \geq L_A^s\vth'(x_0)$
(resp. $L_A^s\vth(x_0) \leq L_A^s\vth'(x_0)$) for all matrices $A\in\mathpzc{M}$,
which implies that
\[
\D\vth(x_0) \geq \D\vth'(x_0)
\]
(resp. $\D \vth(x_0) \leq \D\vth'(x_0)$), see \eqref{eq:MAs}.
Therefore, if $\D w\geq w-\phi$ (resp. $\D w\leq w-\phi$) in $\R^n$,
then, in order to check the viscosity solution condition from Definition \ref{defn:viscosity solution},
we can always restrict ourselves to working in a smaller neighborhood $\mathcal{N}'\subset\mathcal{N}$
containing $x_0$.
\end{rem}

From the definition of $\D$, we see that
\begin{enumerate}
\item If $\tau_hw(x)=w(x+h)$, for some $h\in\R^n$, then $\D(\tau_hw)=\tau_h(\D w)$.
\item For any constant $c\in\R$, $\D(w+c)=\D w$.
\item $\D$ is a concave operator in the sense that, for any $w_1,w_2$,
\[
\D\bigg(\frac{w_1+w_2}{2}\bigg)\geq\frac{1}{2}\D w_1+\frac{1}{2}\D w_2.
\]
\end{enumerate}

Let $w$ be a viscosity subsolution
(resp. supersolution) as in Definition \ref{defn:viscosity solution}. In the next lemma, we state that if $w$ can be touched from above (resp. below)
by a $C^2$ function at a point $x$, then $\D w(x)$ can be computed classically.
This is an important, typical feature of nonlocal equations, see also \cite[Lemma~3.3]{Caffarelli-Silvestre-CPAM}.

\begin{lem}[see~{\cite[Lemma~2.2]{Caffarelli-Charro}}]\label{lem:touchandcompute}
Let $w:\R^n\to\R$ be asymptotically linear at infinity. If
$$\D w \geq w - \phi\quad\big(\hbox{resp.}~\D w \leq w - \phi\big)\quad\hbox{in}~\mathcal{O}\subset\R^n$$
in the viscosity sense and $w$ can be touched by a $C^2$ function
from above (resp. below) at a point $x\in\mathcal{O}$,
then $\delta(w,x,y)/|A^{-1}y|^{n+2s}\in L^1(\R^n)$ for every $A\in\mathpzc{M}$ and
\[
\D w(x)\geq w(x)-\phi(x)\quad\big(\hbox{resp.}~\D w(x)\leq w(x)-\phi(x)\big)
\]
in the classical sense.
\end{lem}

Finally, we recall the comparison principle proved by Caffarelli and Charro.

\begin{thm}[see {\cite[Theorem~4.1]{Caffarelli-Charro}}]
\label{thm:comparison principle}
Let $w_1\in\USC(\R^n)$ and $w_2\in\LSC(\R^n)$ such that
\[
\begin{cases}
\D w_1\geq w_1-\phi\quad\hbox{in}~\R^n\\
\displaystyle\lim_{|x|\to\infty}(w_1-\phi)(x)=0
\end{cases}
\text{and}\quad
\begin{cases}
\D w_2\leq w_2-\phi\quad\hbox{in}~\R^n\\
\displaystyle\lim_{|x|\to\infty}(w_2-\phi)(x)=0.
\end{cases}
\]
Then,
\[
w_1\leq w_2\quad\hbox{in}~\R^n.
\]
\end{thm}

\subsection{The truncated fractional Monge--Amp\`ere operator}\label{subsection:truncated operator}

For $\vep>0$, consider the class
\[
\mathpzc{M}_\vep = \big\{ A \in \mathpzc{M} : \langle A\xi,\xi\rangle \geq \vep|\xi|^2 \text{ for all } \xi \in \R^n\big\}.
\]
Since the matrices in $\mathpzc{M}$ have determinant one, not only are the eigenvalues of $A \in \mathpzc{M}_\vep$ bounded from below, but they are also bounded from above.
In particular,
\[
\vep|\xi|^2\leq \langle A\xi,\xi\rangle\leq \vep^{1-n}|\xi|^2 \quad
\hbox{for every}~\xi\in\R^n.
\] 
We define
\[
\D^\vep u(x)=\inf_{A\in\mathpzc{M}_\vep}L_A^su(x).
\]
The kernels of $L_A^s$, for $A\in\mathpzc{M}_\vep$, satisfy
\begin{equation}
\label{eq:minmaxeigenvalues}
\frac{\vep^{n+2s}}{|y|^{n+2s}}\leq\frac{1}{|A^{-1}y|^{n+2s}}\leq\frac{\vep^{(1-n)(n+2s)}}{|y|^{n+2s}}.
\end{equation}
Therefore, the truncated operator $\D^\vep$ is uniformly elliptic
in the sense of Caffarelli--Silvestre (see Lemma \ref{lem:continuityofLAs} and
\cite[Definition~3.1,~Lemma 3.2]{Caffarelli-Silvestre-CPAM}).
We define the notion of viscosity subsolution, supersolution, and solution for $\D^\vep$ exactly as in
Definition \ref{defn:viscosity solution}. Moreover, by \cite[Lemma~3.3]{Caffarelli-Silvestre-CPAM},
Lemma \ref{lem:touchandcompute} also holds for $\D^\vep$ in place of $\D$.
Obviously, $\D^0=\D$.

Caffarelli and Charro proved in \cite[Theorem~3.1]{Caffarelli-Charro} that the operator $\D$ becomes uniformly elliptic
provided that $\D w$ is bounded below away from zero and $w$ is globally Lipschitz
and semiconcave. The next statement presents an important refinement that will be crucial to proving our results.

\begin{thm}\label{thm:our-ellipticity}
Let $\eta_0$, $L$, and $C$ be positive constants, and fix an open set $\mathcal{O} \subset \R^n$.
There exists $\lambda=\lambda(n,s,\eta_0,L,C)>0$ such that for any
Lipschitz and semiconcave function $w$ with constants $L$ and $C$, respectively, if $0\leq\vep<\lambda$ and
\[
\D^\vep w\geq\eta_0>0\quad\hbox{in}~\mathcal{O}
\]
in the viscosity sense, then
\[
\D^\vep w(x)=\D^\lambda w(x)\quad\hbox{for every}~x\in\mathcal{O}
\]
in the classical sense.
\end{thm}

\begin{proof}
The proof follows by precisely tracking the constants in the proof of Theorem 3.1 in \cite{Caffarelli-Charro}.
Following their notation, fix
\[
\epsilon=\bigg(\frac{\eta_0}{2C_1}\bigg)^{\frac{n-1}{2s}}\quad\hbox{and}
\quad0<\theta<\bigg(\frac{\mu_0}{n\mu_1}\bigg)^{\frac{n-1}{2s}}
\]
(this value of $\epsilon$ is not to be confused with $0\leq\vep<\lambda$ in our hypotheses).
Then, $\epsilon$ and $\theta$ depend only on $n,s,\eta_0,L$, and $C$. Choose $\lambda=\min\{\epsilon,\theta,1\}$.
We notice that if $0\leq \vep<\lambda$, then we can apply \cite[Lemma~3.9]{Caffarelli-Charro}
to deduce Proposition 3.3 in \cite{Caffarelli-Charro}
with $\D^\vep$ in place of $\D$.
Thus, the statement of \cite[Proposition~3.5]{Caffarelli-Charro},
being a simple consequence of Proposition 3.3,
holds in our setting.
We end our proof exactly as in the proof of Theorem 3.1 in \cite[pp.~12-13]{Caffarelli-Charro},
in which $\theta=1/k$.
\end{proof}

We close this section with some continuity and stability results,
whose proofs follow as in \cite[Propositions~2.4~and~2.6]{Silvestre-CPAM} and 
\cite[Lemma~4.5]{Caffarelli-Silvestre-CPAM} by using that $w$ satisfies \eqref{eq:Ls} instead of being just bounded.

\begin{lem}\label{lem:continuityofLAs}
Let $\mathcal{O}$ be an open set, $1/2<s<1$, and $w\in L^1_{\mathrm{loc}}(\R^n)$ be such that
\begin{equation}\label{eq:Ls}
\int_{\R^n}\frac{|w(x)|}{(1+|x|)^{n+2s}}\,dx<\infty.
\end{equation}
Suppose that $w\in C^{1,2s+\mu-1}(\mathcal{O})$ for some $\mu>0$.
Then, for any positive definite symmetric matrix $A$ of size $n\times n$, $L_A^sw\in C^{\mu}(\mathcal{O})$ and
\[
[L_A^sw]_{C^\mu(\mathcal{O})}\leq C[\nabla w]_{C^{2s+\mu-1}(\mathcal{O})},
\]
where $C > 0$ depends only on $n,s,\mu$
and the largest eigenvalue of $A$. In particular, if $\vep>0$, then
\begin{equation}
\label{eqn: equicty}
\hbox{the family}~\{L_A^sw:A\in\mathpzc{M}_{\vep}\}~\hbox{is equicontinuous in}~\mathcal{O}.
\end{equation}
Consequently, by taking the infimum over $A\in\mathpzc{M}_{\vep}$ above,
\[
\D^\vep w\in C(\mathcal{O}).
\]
\end{lem}

We say that a sequence $w_k\in\LSC(\R^n)$, $k\geq1$, $\Gamma$-converges to $w$ in a set $\mathcal{O}$ if the following two conditions hold:
\begin{itemize}
\item[-] For any sequence $x_k\to x$ in $\mathcal{O}$, $\liminf_{k\to\infty}w_k(x_k)\geq w(x)$.
\item[-] For any $x\in\mathcal{O}$, there is a sequence $x_k\to x$ in $\mathcal{O}$ so that $\limsup_{k\to\infty}w_k(x_k)=w(x)$.
\end{itemize}

\begin{lem}\label{lem:gammaconvergence}
Let $w_k\in\LSC(\R^n)\cap L^1_{\mathrm{loc}}(\R^n)$ be a sequence of functions such that
\[
\int_{\R^n}\frac{|w_k(x)|}{(1+|x|)^{n+2s}}\,dx\leq C<\infty
\]
for all $k\geq1$. Let $I$ be either $L_A^s$, for $A\in\mathpzc{M}$, or $\D^\vep$, for $\vep>0$,
for any $1/2<s<1$. Suppose that
\begin{itemize}
\item[-] $Iw_k\leq f_k$ in $\mathcal{O}$;
\item[-] $w_k\to w$ in the $\Gamma$ sense in $\mathcal{O}$;
\item[-] $w_k\to w$ a.e. in $\R^n$; and
\item[-] $f_k\to f$ locally uniformly in $\mathcal{O}$ for some continuous function $f$.
\end{itemize}
Then,
\[
Iw\leq f\quad\hbox{in}~\mathcal{O}.
\]
\end{lem}

An analogous statement to Lemma \ref{lem:gammaconvergence} holds for subsolutions.

\section{Proof of Theorem \ref{thm: solution of obstacle problem}}
\label{Section:Existence}

To construct the solution $u$, we define the class
\begin{equation}
\label{eq:F}
\mathpzc{F}=\Big{\{} w\in \USC(\R^n):\D w\geq w-\phi~\hbox{in}~\R^n,
\,w\leq\psi,\,\hbox{and}\,\lim_{|x|\to\infty}(w-\phi)(x)\leq0 \Big{\}}.
\end{equation}
Notice that $\mathpzc{F}$ is nonempty
because $\phi\in\mathpzc{F}$.
Indeed, by assumption, $\phi<\psi$ in $\R^n$, and by convexity, $\delta(\phi,x,y)\geq0$ for every $x,y\in\R^n$.
Hence, $L_A^s\phi\geq0$ in $\R^n$ for every $A\in\mathpzc{M}$, which implies that $\D\phi\geq0=\phi-\phi$ on $\R^n$.

Now, define
\begin{equation}
\label{eq: u}
u(x)=\big(\sup\big\{w(x):w\in\mathpzc{F}\big\}\big)^\ast \quad\hbox{for}~x\in\R^n.
\end{equation}
By construction,
\[
u\in \USC(\R^n), \quad \phi \leq u\leq\psi, \quad\hbox{and}\quad \lim_{|x|\to\infty}(u-\phi)(x)=0.
\]
In particular,
\[
\int_{\R^n}\frac{|u(x)|}{(1+|x|)^{n+2s}}\,dx<\infty.
\]
Moreover, since $u-\psi$ is upper semicontinuous in $\R^n$, we have that
\begin{equation}
\label{eq:non coincidence set}
\hbox{the noncoincidence set}~\{u<\psi\}~\hbox{is open}.
\end{equation}

First, we will show that $u$, as defined in \eqref{eq: u}, is in the class $\mathpzc{F}$, see Lemma \ref{lem: u subsolution}.
We start with two lemmas.

\begin{lem}
\label{lem: max of subsolutions}
Let $w_1, w_2 \in \mathpzc{F}$.
Then,
\[
w(x)= \max\{w_1(x),w_2(x)\} \in \mathpzc{F}.
\]
\end{lem}

\begin{proof}
Evidently, $w\in\USC(\R^n)$, $w \leq \psi$, and $\lim_{|x| \to \infty} (w-\phi)(x) \leq 0$.
Let $P$ be a $C^2$ function touching $w$ from above at $x_0$ in $\mathcal{N}$.
Without loss of generality, $P(x_0)=w(x_0)=w_1(x_0)$, so $P$ also touches $w_1$
from above at $x_0$ in $\mathcal{N}$. 
Let
\[
\vth(x)=
\begin{cases}
P(x)&\hbox{if}~x\in \mathcal{N}\\
w(x)&\hbox{if}~x\in\R^n\setminus \mathcal{N}
\end{cases}
\quad\text{and}\quad
\vth_1(x)=
\begin{cases}
P(x)&\hbox{if}~x\in \mathcal{N}\\
w_1(x)&\hbox{if}~x\in\R^n\setminus \mathcal{N}.
\end{cases}
\]
Observe that $\vth(x_0)=\vth_1(x_0)$ and $\vth \geq \vth_1$ in $\R^n$, from which it follows that
$\delta(\vth,x_0,y) \geq \delta(\vth_1,x_0,y)$ for any $y\in\R^n$.
Therefore, given any matrix $A\in\mathpzc{M}$,
\[
L_A^s\vth(x_0)\geq L_A^s \vth_1(x_0) \geq \D \vth_1(x_0) \geq \vth_1(x_0)-\phi(x_0)=\vth(x_0)-\phi(x_0).
\]
Thus, $\D\vth(x_0)\geq \vth(x_0)-\phi(x_0)$, and $w$ is a subsolution to $\D w\geq w-\phi$ in $\R^n$.
\end{proof}

\begin{lem}
\label{lem: gamma}
Let $u$ be as in \eqref{eq: u} and let $P$ be a $C^2$ function touching
$u$ from above at $x_0$ in $\Nhd$.
Given any open neighborhood $\Nhd'\subset\subset\Nhd$ that contains $x_0$,
there exist functions $u_k \in \mathpzc{F}$, points $x_k \in \Nhd'$, and constants $d_k>0$, for $k\geq1$, such that
\[
u_k \leq u_{k+1},\quad x_k\to x_0,\quad d_k\searrow0,\quad u_k(x_k)\to u(x_0),
\]
and
\begin{equation}
\label{eq:Pky}
P_k(y)=P(y)+\frac{|y-x_k|^2}{k}-d_k\quad\hbox{touches}~u_k~\hbox{from above at}~x_k~\hbox{in}~\Nhd'.
\end{equation}
\end{lem}

\begin{proof}
Fix $x_0\in\R^n$. The proof is divided into two steps.

\medskip

\noindent\textit{-- Step 1.} There exist points $y_k$ and functions
$u_k \in \mathpzc{F}$ with $u_{k} \leq u_{k+1}$ such that
\[
y_k\to x_0\quad\hbox{and}\quad u_k(y_k)\to u(x_0).
\]
Indeed, by Remark~\ref{rmk: property 1 sc envelope}, there exists a sequence of points $y_k$
such that
\begin{equation}\label{a}
y_k\to x_0\quad\hbox{and}\quad\overline{w}(y_k)\equiv\sup_{w\in\mathpzc{F}}w(y_k)\to u(x_0).
\end{equation}
Let $k\geq1$. There is a sequence $\{w_{k,j}\}_{j=1}^\infty\in\mathcal{F}$ 
such that
\begin{equation}
\label{c}
w_{k,j}(y_k)\nearrow \overline{w}(y_k)\quad\hbox{as}~j\to\infty.
\end{equation}
In particular, there exists $J(k)>0$ such that
\begin{equation}\label{b}
0 \leq \overline{w}(y_k) - w_{k,j}(y_k) <1/k\quad\hbox{for every}~j\geq J(k).
\end{equation}
Without loss of generality we can let $J(k)<J(k+1)$, for every $k\geq1$.
Define
\[
u_k(y)= \max\big\{ w_{1,J(1)}(y), \ldots, w_{k,J(k)}(y)\big\} \quad\hbox{for}~y\in\R^n.
\]
Then, $u_k \leq u_{k+1}$ and, by Lemma~\ref{lem: max of subsolutions},
$u_k \in \mathcal{F}$ for every $k\geq1$.
Finally, observe that, by the definition of $u_k$, \eqref{c}, \eqref{a}, and \eqref{b}, as $k\to\infty$,
\begin{align*}
|u(x_0)-u_k(y_k)| &\leq |u(x_0)-\overline{w}(y_k)|+\big(\overline{w}(y_k) - u_k(y_k)\big) \\
&\leq |u(x_0)-\overline{w}(y_k)|+(\overline{w}(y_k)-w_{k,J(k)}(y_k)) \to0.
\end{align*}

\medskip

\noindent\textit{-- Step 2.} Let $\Nhd'\subset\subset\Nhd$ be any open neighborhood of $x_0$.
Without loss of generality, we can assume that the sequence $y_k$ from Step 1
satisfies $y_k \in \Nhd'$ for all $k\geq1$. 
Define 
\[
d_k = \inf_{\overline{\Nhd'}} (P - u_k).
\]
Notice that $d_k \geq 0$ is well defined because $P-u_k$ is lower semicontinuous in $\R^n$. 
Moreover, $d_k\geq d_{k+1}$ as $u_k \leq u_{k+1} \leq u \leq P$ in $\Nhd$.
Also,
\[
0\leq d_k\leq P(y_k) - u_k(y_k)\to P(x_0)-u(x_0) = 0.
\]
Let $x_k\in\overline{\Nhd'}$ be such that
\begin{equation}
\label{eq: dk}
P(x_k)-u_k(x_k)=d_k.
\end{equation}
The set of points $\{x_k\}_{k=1}^\infty$ is bounded, so, after passing to a subsequence, we can assume
that $\{x_k\}_{k=1}^\infty$ is convergent in $\overline{\Nhd'}$.

Let us show that $x_k\to x_0$. Suppose, to the contrary, that there exists
a subsequence $\{x_{k_j}\}_{j=1}^\infty$ of $\{x_k\}_{k=1}^\infty$
such that $x_{k_j}\to x'\in\overline{\Nhd'}$, as $j\to\infty$, with $x'\neq x_0$.
Then, as $P>u$ in $\overline{\Nhd'}\setminus\{x_0\}$ and $P-u$ is lower semicontinuous,
\[
0<P(x')-u(x')\leq\liminf_{j\to\infty}(P-u)(x_{k_j})\leq\lim_{j\to\infty}d_{k_j}=0,
\]
which is a contradiction. 
Hence, $x_k\to x_0$, as desired.

This and \eqref{eq: dk} imply that $u_k(x_k)\to u(x_0)$.
By construction, $P(x_k)-d_k=u_k(x_k)$ and $P-d_k\geq u_k$ in $\Nhd'$.
So, $P_k(y)$ as defined in \eqref{eq:Pky} is a $C^2$ function that touches $u_k$ from above at $x_k$ in $\Nhd'$.
\end{proof}

\begin{rem}
In Lemma~\ref{lem: gamma}, we can modify the definition of $P_k(y)$ in \eqref{eq:Pky}.
Indeed, as the proof above shows, any function of the form
\[
P(y)+\varphi(y)-d_k,
\]
where $\varphi$ is a $C^2$ function such that $\varphi(x_k)=0$ and  $\varphi(y)>0$ 
for all $y\in\overline{\Nhd'}\setminus\{ x_k\}$, will touch $u_k$ from above at $x_k$ in $\Nhd'$.
\end{rem}

\begin{lem}
\label{lem: u subsolution}
Let $u$ be as in \eqref{eq: u}.
Then,
\[
\D u\geq u-\phi\quad\hbox{in}~\R^n.
\]
In particular,
\[
u\leq\bar{u},
\]
where $\bar{u}$ is the solution to \eqref{eq:v}.
\end{lem}

\begin{proof}
Let $P$ be a $C^2$ function touching $u$ from above at $x_0$ in $\mathcal{N}$.
By Lemma~\ref{lem: gamma}, there exist functions $u_k \in \mathpzc{F}$, points $x_k\in B_r(x_0)\subset\subset\Nhd$ for some $r > 0$, and constants $d_k>0$ such that $u_k \leq u_{k+1}\leq u$, $d_k\searrow0$, $u_k(x_k)\to u(x_0)$, and $P_k(y)=P(y)+\frac{1}{k}|y-x_k|^2-d_k$ touches $u_k$ from above at $x_k$ in $B_r(x_0)$ for $k\geq1$. 
Define the test functions
\[
\vth(x)=
\begin{cases}
P(x)&\hbox{if}~x\in B_r(x_0)\\
u(x)&\hbox{if}~x\in\R^n\setminus B_r(x_0)
\end{cases}
\quad\text{and}\quad
\vth_k(x)=
\begin{cases}
P_k(x)&\hbox{if}~x\in B_r(x_0)\\
u_k(x)&\hbox{if}~x\in\R^n\setminus B_r(x_0).
\end{cases}
\]
We recall that, by Remark~\ref{rmk: viscosity nhd}, it is enough to use $\vth$ as defined above as a test function for $u$. 
Let $A\in\mathpzc{M}$ and let $\Lambda_A$ denote the maximum eigenvalue of $A$. 
Then,
\begin{align*}
c_{n,s}^{-1}&L_A^s\vth_k(x_k) \\
&= \lim_{\rho\to0}\int_{B_r(x_0)\setminus B_\rho(x_k)}
\frac{\vth_k(y)-\vth_k(x_k)}{|A^{-1}(y-x_k)|^{n+2s}}\,dy
+\int_{\R^n\setminus B_r(x_0)}\frac{\vth_k(y)-\vth_k(x_k)}{|A^{-1}(y-x_k)|^{n+2s}}\,dy \\
&\leq \lim_{\rho\to0}\bigg[\int_{B_r(x_0)\setminus B_\rho(x_k)}
\frac{P(y)-P(x_k)}{|A^{-1}(y-x_k)|^{n+2s}}\,dy
+\frac{1}{k} \int_{B_r(x_0)\setminus B_\rho(x_k)}\frac{|y-x_k|^2}{|A^{-1}(y-x_k)|^{n+2s}}\,dy\bigg] \\
&\quad+\int_{\R^n\setminus B_r(x_0)}\frac{u(y)-P(x_k)}{|A^{-1}(y-x_k)|^{n+2s}}\, dy
+\int_{\R^n \setminus B_r(x_0)} \frac{d_k}{|A^{-1}(y-x_k)|^{n+2s}}\, dy\\
&\leq c_{n,s}^{-1}L_A^s\vth(x_k)+\Lambda_A^{n+2s}
\bigg[\frac{1}{k}\int_{B_r(x_0)}\frac{1}{|y-x_k|^{n+2s-2}}\,dy
+\int_{\R^n \setminus B_r(x_0)}\frac{d_k}{|y-x_k|^{n+2s}}\,dy\bigg]   \\
&\leq c_{n,s}^{-1}L_A^s\vth(x_k)+C\big(k^{-1}+d_k\big),
\end{align*}
where $C=C(n,s,\Lambda_A,r)>0$ is independent of $k$. 
Here, we have used that
\[
\int_{\R^n\setminus B_r(x_0)}\frac{1}{|y-x_k|^{n+2s}}\,dy\to
\int_{\R^n\setminus B_r(x_0)}\frac{1}{|y-x_0|^{n+2s}}\,dy
\]
and
\[
\int_{B_r(x_0)}\frac{1}{|y-x_k|^{n+2s-2}}\,dy\to
\int_{B_r(x_0)}\frac{1}{|y-x_0|^{n+2s-2}}\,dy,
\]
as $k\to\infty$.
Hence, as $u_k$ is a subsolution,
\[
\vth_k(x_k) - \phi(x_k) \leq \D \vth_k(x_k) \leq L_A^s\vth(x_k) +C\big(k^{-1}+d_k\big).
\]
Notice that $\vth_k(x_k)=P_k(x_k)\to P(x_0)=\vth(x_0)$ as $k\to\infty$.
Together with Lemma \ref{lem:continuityofLAs}, this implies that
\[
\vth(x_0) - \phi(x_0) \leq L_A^s \vth(x_0).
\]
Since $A\in\mathpzc{M}$ was arbitrary, we obtain $\vth(x_0)-\phi(x_0)\leq\D\vth(x_0)$, which means that $u$ is a subsolution to $\D u\geq u-\phi$.

We have already seen that $\lim_{|x|\to\infty}(u-\phi)(x)=0$.
Thus, the comparison principle (Theorem~\ref{thm:comparison principle}) implies that $u\leq\bar{u}$.
\end{proof} 

With Lemma~\ref{lem: u subsolution} in hand, we can prove that the contact set $\{ u = \psi \}$
is compact and that $u$ is Lipschitz and semiconcave with constants no larger than those of $\phi$ and $\psi$.

\begin{lem}\label{lem:compact}
Let $u$ be as in \eqref{eq: u}. 
Then, 
\[
\{ u = \psi \}~\hbox{is compact}.
\]
\end{lem}

\begin{proof}
We know that $u\leq\psi$ and that the noncoincidence set $\{u<\psi\}$
is open, see \eqref{eq:non coincidence set}. 
Therefore, the contact set $\{u=\psi\}$ is closed.
On the other hand, by Lemma \ref{lem: u subsolution}, $\{\bar{u}<\psi\}\subset\{u<\psi\}$,
which implies that $\{u=\psi\}\subset \mathcal{K}$.
Hence, the contact set is compact.
\end{proof}

Recall the definition of $M_1$ and $M_2$ from the statement of Theorem \ref{thm: solution of obstacle problem}.

\begin{lem}
\label{lem: u Lip and SC}
Let $u$ be as in \eqref{eq: u}. 
Then, $u$ is Lipschitz continuous and semiconcave with
\[
[u]_{\Lip(\R^n)} \leq M_1 \quad\hbox{and}\quad\SC(u)\leq M_2.
\]
\end{lem}

\begin{proof}
Given any $h\in\R^n$, let us first show that
\begin{equation}
\label{eq:winF}
w(x)=u(x+h)-M_1|h|\in\mathpzc{F}.
\end{equation}
Indeed, $w\in\USC(\R^n)$,
\[
\lim_{|x|\to\infty}(w-\phi)(x)=\lim_{|x|\to\infty}\Big[\big(u(x+h)-u(x)\big)
+\big(u(x)-\phi(x)\big)\Big]-M_1|h|\leq0,
\]
and since $-M_1|h|\leq\psi(x)-\psi(x+h)$ and $u\leq\psi$,
\[
w(x)=u(x+h)-M_1|h|\leq u(x+h)-\psi(x+h)+\psi(x)\leq\psi(x).
\]
Finally, as $\D$ is translation invariant, $\D c=0$ for any constant $c$,
 and $\phi(x+h)-\phi(x)\leq M_1|h|$, we find that
\begin{align*}
\D w(x) &= \D(\tau_hu)(x)=(\D u)(x+h)\geq u(x+h)-\phi(x+h) \\
&= (u(x+h)-M_1|h|)-\phi(x+h)+M_1|h| \\
&\geq w(x)-\phi(x),
\end{align*}
in the viscosity sense. 
Thus, \eqref{eq:winF} is proved.
Now, by the maximality of $u$ in $\mathpzc{F}$, $w\leq u$, which means that
\[
u(x+h)-u(x)\leq M_1|h|.
\]
Since $x$ and $h$ above are arbitrary, we conclude that $[u]_{\Lip(\R^n)} \leq M_1$.

Given any $h\in\R^n$, let us first see that
\begin{equation}
\label{eq:omegainF}
w(x) = \frac{u(x+h)+u(x-h)-M_2|h|^2}{2} \in \mathpzc{F}.
\end{equation}
Indeed, $w\in\USC$, and since $\delta(\phi,x,h)\leq M_2|h|^2$,
\begin{align*}
(w-\phi)(x) &= \frac{u(x+h)+u(x-h)}{2}-\frac{M_2|h|^2}{2}-\phi(x)\\
& = \frac{(u-\phi)(x+h)+(u-\phi)(x-h)}{2}+\frac{\delta(\phi,x,h)-M_2|h|^2}{2} \\
&\leq \frac{(u-\phi)(x+h)+(u-\phi)(x-h)}{2}\to0
\end{align*}
as $|x|\to\infty$. 
Also, $u\leq\psi$ and $\delta(\psi,x,h)\leq M_2|h|^2$, which implies that
\[
w(x)=\frac{u(x+h)+u(x-h)}{2}-\frac{M_2|h|^2}{2}
\leq \frac{\psi(x+h)+\psi(x-h)}{2}-\frac{M_2|h|^2}{2} \leq \psi(x).
\]
Finally, using the inequality $M_2|h|^2-\delta(\phi,x,h)\geq0$,
\begin{align*}
\D w(x) &\geq \frac{1}{2}\D(\tau_hu)(x)+\frac{1}{2}\D(\tau_{-h}u)(x)\\
&= \frac{1}{2}\D u(x+h) + \frac{1}{2}\D u(x-h)\\
&\geq \frac{(u-\phi)(x+h)+(u-\phi)(x-h)}{2}\\
&= w(x)+\frac{M_2|h|^2-\delta(\phi,x,h)}{2}-\phi(x)\geq w(x)-\phi(x),
\end{align*}
in the viscosity sense. 
Thus, \eqref{eq:omegainF} is proved. 
By the maximality of $u$ in $\mathpzc{F}$, we have that $w\leq u$. 
Hence,
\[
u(x+h)+u(x-h)-2u(x)\leq M_2|h|^2
\]
or, equivalently, $u$ is semiconcave and $\SC(u) \leq M_2$.
\end{proof}

The semiconcavity of $u$ permits us to compute $\D u(x)$ in the classical sense,
see Lemma \ref{lem:touchandcompute}. We use this to show that $\D u(x)$ is bounded from above.

\begin{lem}
\label{lem:classical soln}
Let $u$ be as in \eqref{eq: u}. 
Then, $\D u(x)$ can be computed in the classical sense and
\[
0\leq \D u(x)\leq C\big(1 + \|u-\phi\|_{L^\infty(\R^n)}\big)\quad\hbox{for every}~x\in\R^n,
\]
for some constant $C=C(n,s,a,\epsilon,M_2)>0$.
\end{lem}

\begin{proof}
As $u$ is semiconcave on $\R^n$ (see Lemma \ref{lem: u Lip and SC}), it can be touched from above by a $C^2$
function at every point $x \in \R^n$.
Thus, Lemmas \ref{lem: u subsolution} and \ref{lem:touchandcompute} imply that $\D u(x)$ can be computed classically and 
$\D u(x)\geq u(x)-\phi(x)\geq0$ for every $x\in\R^n$.
Since, for any $x\in\R^n$, we have $\delta(u,x,y)/|y|^{n+2s}\in L^1(\R^n)$
and $\delta(\phi,x,y)\geq0$, we can estimate
\begin{align*}
\D u(x) &\leq -(-\Delta)^s u(x) \\
&= c_{n,s}\int_{B_1} \frac{\delta(u,x,y)}{|y|^{n+2s}}\,dy
+c_{n,s}\int_{\R^n \setminus B_1} \frac{\delta(u,x,y)}{|y|^{n+2s}}\,dy\\
&\leq c_{n,s}\int_{B_1} \frac{\SC(u)|y|^2}{|y|^{n+2s}} \,dy
+c_{n,s}\int_{\R^n\setminus B_1}\frac{\delta(u-\phi,x,y)}{|y|^{n+2s}}\,dy
+c_{n,s}\int_{\R^n \setminus B_1} \frac{\delta(\phi,x,y)}{|y|^{n+2s}}\,dy \\
&\leq C\big(M_2 +\|u-\phi\|_{L^\infty(\R^n)} -(-\Delta)^s\phi(x)\big) \\
&\leq C\big(1+\|u-\phi\|_{L^\infty(\R^n)}\big).
\end{align*}
In the last inequality, we have used that
\begin{equation}\label{M0}
0\leq -(-\Delta)^s\phi\leq M_0\quad\hbox{in}~\R^n,
\end{equation}
for some constant $M_0=M_0(n,s,a,\epsilon)>0$, see \cite[eq.~(6.8)]{Caffarelli-Charro}.
\end{proof}

Next we need to consider
the obstacle problem \eqref{eq:obstacle problem} for the truncated fractional Monge--Amp\`ere operator
defined in subsection~\ref{subsection:truncated operator}.

\begin{thm}
\label{thm: eps-obstacle problem}
For any $\varepsilon>0$, there exists a unique classical solution $u_\varepsilon$ to the obstacle problem
\[
\begin{cases}
\D^\vep u_\vep\geq u_\vep-\phi&\hbox{in}~\R^n\\
u_\vep\leq\psi&\hbox{in}~\R^n\\
\D^\vep u_\vep=u_\vep-\phi&\hbox{in}~\{u_\vep<\psi\}\\
\displaystyle\lim_{|x|\to\infty}(u_\vep-\phi)(x)=0.
\end{cases}
\]
Moreover, $u_\vep$ is Lipschitz and semiconcave with constants no larger than $M_1$ and $M_2$, respectively.
\end{thm}

\begin{proof}
Fix $\vep>0$. Parallel to \eqref{eq:F}, we define the class
\[
\mathpzc{F}_\vep = \Big\{w\in \USC(\R^n):\De w\geq w-\phi~\hbox{in}~\R^n,
\,w\leq\psi,\,\hbox{and}\,\lim_{|x|\to\infty}(w-\phi)(x)\leq0\Big\}.
\]
Then, $\phi\in\mathpzc{F}_\varepsilon$. By replacing $u$ by $u_\vep$ and $\D$ by $\De$
in the arguments of Lemmas~\ref{lem: max of subsolutions}--\ref{lem: u subsolution}
and Lemma~\ref{lem: u Lip and SC}, we deduce that
\[
u_\vep(x)=\big(\sup\big\{w(x):w\in\mathpzc{F}_\vep\big\}\big)^\ast \quad\hbox{for}~x\in\R^n
\]
is the largest function in $\mathpzc{F}_\vep$, $\phi\leq u_\varepsilon\leq\psi$, $\lim_{|x|\to\infty}(u_\varepsilon-\phi)(x)=0$,
\begin{equation}\label{eq: lip u_eps}
[u_\vep]_{\Lip(\R^n)}\leq M_1\quad\hbox{and}\quad\SC(u_\varepsilon)\leq M_2,
\end{equation}
and the noncoincidence set $\{u_\varepsilon<\psi\}$ is open.

It remains to prove that
\[
\De u_\vep = u_\vep - \phi \quad\hbox{in}~\{u_\vep < \psi\}.
\]
We argue by contradiction. 
Specifically, we will show that if $\De u_\vep = u_\vep - \phi$ fails in the open set $\{u_\vep < \psi\}$, then $u_\vep$ is not maximal in $\mathpzc{F}_\vep$.
To this end, let $x_0\in \{u_\vep < \psi \}$ and $P$ be a $C^2$ function
touching $u_\vep$ from below at $x_0$ in $\overline{\Nhd}$ such that for
\[
\vth(x)=
\begin{cases}
P(x)&\hbox{for}~x\in \Nhd\\
u_\vep(x)&\hbox{for}~x\in\R^n\setminus \Nhd,
\end{cases}
\]
we have
\[
\De \vth(x_0)>\vth(x_0)-\phi(x_0).
\]
Recall that $\De \vth$ is continuous in $\Nhd$ (see Lemma~\ref{lem:continuityofLAs}).
Hence, given 
\[
0<\tau<\De \vth(x_0)-(\vth(x_0)-\phi(x_0)),
\]
there exists a ball $B_r(x_0)\subset\subset\Nhd\cap\{u_\vep<\psi\}$ such that
\begin{equation}
\label{Delta}
\De \vth(z) \geq \vth(z) - \phi(z) + \tau\quad\hbox{for every}~z\in B_r(x_0).
\end{equation}
Next, we lift $P$ in $\mathcal{N}$ by a small amount $d>0$
(to be fixed) so that $\{u_\vep<P+d\} \subset \subset B_r(x_0)$ and $\{ P+d<\psi \} \subset \{ u_\vep < P + d\}$.
We then set
\[
u'_\vep(x)=
\begin{cases}
P(x)+d&\hbox{if}~x\in\{u_\vep<P+d\}\\
u_\vep(x)&\hbox{otherwise},
\end{cases}
\]
and notice that $u_\vep'$ is continuous, $u_\vep'\geq u_\vep$ in $\R^n$, and $\lim_{|x|\to \infty}(u'_\vep-\phi)(x) = 0$.
If we can show that $u_\vep'$ is a subsolution to $\De w = w - \phi$, then $u_\vep$ is not maximal in $\mathpzc{F}_\vep$
because we constructed $u_\vep'$ in such a way that
\[
u_\vep'(x_0)=P(x_0)+d>P(x_0)=u_\vep(x_0).
\]

Therefore, we now prove that
\begin{equation}\label{one claim}
u'_\vep~\hbox{is a subsolution to}~\De w = w - \phi~\hbox{in}~\R^n.
\end{equation}
Let $P'$ be a $C^2$ function touching $u'_\vep$ from above at $x'$ in $\Nhd'$.
We have two cases to consider.

\smallskip

\noindent\textit{-- Case 1. $x' \in \{ u'_\vep = u_\vep \}$}.
Since $\delta(u'_\vep,x',y) \geq \delta(u_\vep,x',y)$, by Lemma \ref{lem:touchandcompute}
(which is also valid for the uniformly elliptic case), we see that 
\[
\De u'_\vep(x') \geq \De u_\vep(x') \geq u_\vep(x') - \phi(x') = u'_\vep(x') - \phi(x'),
\]
and \eqref{one claim} follows. 

\smallskip

\noindent\textit{-- Case 2. $x' \in \{ u'_\vep > u_\vep \}$}. Define
\[
\vth'(x)=\begin{cases}
P'(x)&\hbox{if}~x\in \Nhd'\\
u'_\vep(x)&\hbox{if}~x\in\R^n\setminus \Nhd'.
\end{cases}
\]
Remark~\ref{rmk: viscosity nhd} allows us to assume that $\Nhd'\subset\{ u'_\vep > u_\vep \}=\{u_\vep<P+d\}\subset\subset B_r(x_0)$.
Observe that $P' - d \geq u'_\vep -d = P$ in $\Nhd'$ and $P'(x')-d=P(x')$. Then,
\begin{equation}
\label{Alpha}
\begin{aligned}
c_{n,s}^{-1}L_A^s\vth'(x')
&= \lim_{\rho\to 0} \int_{\Nhd' \setminus B_\rho(x')}\frac{P'(y)-u'_\vep(x')}{|A^{-1}(y-x')|^{n+2s}}\,dy
+\int_{\R^n \setminus \Nhd'} \frac{u'_\vep(y)-u'_\vep(x')}{|A^{-1}(y-x')|^{n+2s}}\, dy\\
&=\lim_{\rho\to 0}\int_{\Nhd' \setminus B_\rho(x')}\frac{(P'(y)-d)-(P'(x')-d)}{|A^{-1}(y-x')|^{n+2s}}\,dy
+\mathrm{I}\\
&\geq \lim_{\rho\to 0} \int_{\Nhd' \setminus B_\rho(x')} \frac{P(y) - P(x')}{|A^{-1}(y-x')|^{n+2s}}\, dy +  \mathrm{I}.
\end{aligned}
\end{equation}
We estimate the integral $\mathrm{I}$ from below.
Since $u_\vep'\geq u_\vep$ in $\R^n$, $u'_\vep(x')=P'(x')=P(x')+d$, and $u'_\vep \geq P + d$ in $\Nhd \setminus \Nhd'$, we find that
\begin{equation}
\label{Beta}
\begin{aligned}
\mathrm{I} &= \int_{\R^n \setminus \Nhd} \frac{u'_\vep(y) - u'_\vep(x')}{|A^{-1}(y-x')|^{n+2s}}\, dy
+\int_{\Nhd \setminus \Nhd'} \frac{u'_\vep(y) - u'_\vep(x')}{|A^{-1}(y-x')|^{n+2s}}\, dy\\
&\geq \int_{\R^n \setminus \Nhd} \frac{u_\vep(y) - P(x')}{|A^{-1}(y-x')|^{n+2s}}\, dy
-\int_{\R^n \setminus \Nhd} \frac{d}{|A^{-1}(y-x')|^{n+2s}}\, dy\\
&\quad+ \int_{\Nhd \setminus \Nhd'} \frac{P(y) - P(x')}{|A^{-1}(y-x')|^{n+2s}}\, dy.
\end{aligned}
\end{equation}
From \eqref{Alpha}, \eqref{Beta} and the definition of $\vth$, we get
\begin{align*}
c_{n,s}^{-1}L_A^s\vth'(x') &\geq \lim_{\rho\to 0} \int_{\Nhd\setminus B_\rho(x')}\frac{P(y) - P(x')}{|A^{-1}(y-x')|^{n+2s}}\, dy 
+\int_{\R^n \setminus \Nhd} \frac{u_\vep(y) - P(x')}{|A^{-1}(y-x')|^{n+2s}}\, dy \\
&\quad-d\int_{\R^n \setminus \Nhd} \frac{1}{|A^{-1}(y-x')|^{n+2s}}\, dy \\
&= c_{n,s}^{-1}L_A^s\vth(x')-Cd,
\end{align*}
where $C = C(n,s,\vep,\Nhd) > 0$ is independent of $A$.
By taking the infimum over all $A\in\mathpzc{M}_\vep$ above and using \eqref{Delta}, we deduce that
\[
\De \vth'(x') \geq \De \vth(x')-Cd \geq\vth(x')-\phi(x')+\tau-Cd =\vth'(x')-\phi(x')+\tau-(C+1)d.
\]
Thus, by choosing $d > 0$ sufficiently small, it follows that
\[
\De \vth'(x') \geq \vth'(x') - \phi(x').
\]
This completes the proof of \eqref{one claim} and the theorem.
\end{proof}

Let $u_\vep$ and $\mathpzc{F}_\vep$ be as in Theorem~\ref{thm: eps-obstacle problem} and its proof.
First, notice that if $\vep_0\geq\vep$, then $\D^{\vep_0}u_\vep\geq \De u_\vep$.
In particular, $u_\vep \in \mathpzc{F}_{\vep_0}$. 
Hence, by the maximality of $u_{\vep_0}$ in $\mathpzc{F}_{\vep_0}$, we have that $u_{\vep_0}\geq u_{\vep}$.
In other words, the sequence of functions $u_\vep$ is decreasing as $\vep \searrow 0$. 
Let
\begin{equation}
\label{uzero}
u_0(x) = \inf_{\vep > 0} u_\vep(x)\quad\hbox{for}~x\in\R^n.
\end{equation}
Then, $u_0$ is well-defined because $\phi\leq u_\vep\leq\psi$ for every $\vep>0$.
Clearly,
\[
\phi \leq u_0 \leq \psi\quad\text{and}\quad \lim_{|x|\to \infty} (u_0-\phi)(x) = 0.
\]
Moreover, \eqref{eq: lip u_eps} and Arzel\`a--Ascoli's theorem imply that $u_0$ is the local uniform (decreasing) limit of $u_\vep$ and that $u_0$ is Lipschitz continuous with $[u_0]_{\Lip(\R^n)}\leq M_1$.

The following is one of the crucial, most delicate estimates we need.

\begin{lem}\label{lem:uzero separated}
Let $u_0$ be as in \eqref{uzero}. Then,
\begin{equation*}
u_0 > \phi.
\end{equation*}
\end{lem}

\begin{proof}
Let us argue by contradiction. 
Suppose that there is a point $x_0$ such that $u_0(x_0)=\phi(x_0)$. 
Then, as $\phi<\psi$ in $\R^n$, we have $x_0 \in \{u_0<\psi\}$.
Since $\phi\in C^{2,\sigma}(\R^n)$ is strictly convex in compact sets and asymptotically close to a cone at infinity, we can find a function $\varphi\in C^{2,\sigma}(\R^n)$ that is also strictly convex in compact sets, asymptotically close to a cone at infinity, and touches both $u_0$ and $\phi$ from below in $\overline{B_r(x_0)}$ at $x_0$ for some $r>0$. 
As $u_0$ is the local uniform limit of $u_\vep$, there exist points $x_\vep\in B_r(x_0)$ such that $x_\vep\to x_0$ and $u_\vep$ can be touched from below at $x_\vep$ in $B_r(x_0)$ by 
\begin{equation}
\label{eq:varphieps}
\varphi_\vep(x) = \varphi(x) - \vep \omega(|x-x_\vep|) + d_\vep \quad\hbox{for}~x\in\R^n.
\end{equation}
Here, $d_\vep\searrow0$ and $\omega = \omega(t)$ is convex, strictly increasing in $[0,r)$, smooth in $(0,r)$, linear in $\R\setminus[0,r)$, $\omega(0^+)=\omega'(0^+)=0$, and such that $\varphi_\vep$ is strictly convex in compact sets.
Because $\varphi_{\vep}$ is convex and touches the supersolution $u_\vep$  from below at $x_\vep$,
\begin{equation}
\label{eq:casicasi}
0 \leq \D \varphi_\vep(x_\vep) \leq \De \varphi_\vep(x_\vep) \leq \varphi_\vep(x_\vep)-\phi(x_\vep) = u_\vep(x_\vep) - \phi(x_\vep).
\end{equation}
As the sets $\{ u_\vep < \psi \}$ are increasing and $x_\vep\to x_0$, we have $x_\vep\in\{u_\vep<\psi\}$ for all $\vep>0$ sufficiently small.
Moreover,
\begin{align*}
u_\vep(x_\vep) - \phi(x_\vep) \to 0\quad\hbox{as}~\vep\searrow0.
\end{align*}
This and \eqref{eq:casicasi} imply that
\begin{equation}
\label{eq:casicasicasi}
0 \leq \D \varphi_\vep(x_\vep)\to 0\quad\hbox{as}~\vep\searrow0.
\end{equation}
Using this last inequality, we will prove that there exists a direction $\e_0\in\S^{n-1}$ such that
\begin{equation}
\label{eq:fracLaplines}
-(-\Delta)^s_{\e_0}\varphi(x_0)=\frac{c_{n,s}}{2}\int_{\R}\frac{\delta(\varphi,x_0,t\e_0)}{|t|^{1+2s}}\,dt\leq0.
\end{equation}
This clearly contradicts the convexity of the nonconstant function
$\varphi$.
In turn, $u_0 > \phi$, as desired.

To deduce \eqref{eq:fracLaplines}, suppose, to the contrary, that
\begin{equation}
\label{eq:mu}
-(-\Delta)_\e^s \varphi(x_0)\geq \mu > 0 \quad\hbox{for all}~\e \in \S^{n-1}.
\end{equation}
Since the family $\{-(-\Delta)_\e^s \vphi\}_{\e \in \S^{n-1}}$ is equicontinuous (see \eqref{eqn: equicty}),
\[
\inf_{\e \in \S^{n-1}} \big{\{} -(-\Delta)_\e^s \varphi(x_\vep) \big{\}}\geq \frac{\mu}{2} > 0 \quad\hbox{for all $\vep$ sufficiently small}.
\]
The function $\omega\equiv\omega(|\cdot|)$ in \eqref{eq:varphieps} is radially symmetric and convex. Hence,
$(-\Delta)^s_\e\omega(0)$ is a negative constant independent of $\e\in\S^{n-1}$.
Therefore, we can ensure that $\vep(-\Delta)^s_\e\omega(0)\geq-\mu/4$ provided $\vep$ is
sufficiently small, independently of the direction $\e\in\S^{n-1}$.
Collecting these last two facts and \eqref{eq:varphieps}, we deduce that
\begin{equation}
\label{muover4}
-(-\Delta)_\e^s \varphi_\vep(x_\vep) = -(-\Delta)_\e^s \varphi(x_\vep)+\vep(-\Delta)^s_\e\omega(0)
\geq\frac{\mu}{4}>0,
\end{equation}
uniformly in $\e\in\S^{n-1}$ and for all $\vep$ sufficiently small.
It is show in the proof of \cite[Proposition~3.5]{Caffarelli-Charro} that an estimate
of the form \eqref{muover4} readily yields the existence of a positive constant $\tau=\tau(n,s,\mu,[\vphi]_{\Lip(\R^n)},\SC(\vphi))$
such that  $\D \varphi_\vep(x_\vep) \geq \tau$.
This estimate is uniform in $\vep$, a contradiction to \eqref{eq:casicasicasi}.
Thus, \eqref{eq:mu} cannot hold. 
In other words, there are directions $\e_k\in\S^{n-1}$ such that
\begin{equation}
\label{eq:almostthere}
-(-\Delta)_{\e_k}^s \varphi(x_0)
=\frac{c_{n,s}}{2}\int_{\R}\frac{\delta(\varphi,x_0,t\e_k)}{|t|^{1+2s}}\,dt\leq\frac{1}{k},
\end{equation}
for each $k\geq1$. 
The compactness of $\S^{n-1}$ allows us to assume, without loss of generality, that $\e_k\to\e_0$ for some $\e_0\in\S^{n-1}$,
as $k\to\infty$. The continuity of $\varphi$ gives
\[
\frac{\delta(\varphi,x_0,t\e_k)}{|t|^{1+2s}}\to \frac{\delta(\varphi,x_0,t\e_0)}{|t|^{1+2s}},
\]
as $k\to\infty$. Since $\varphi$ is convex and has linear growth at infinity,
\[
0\leq\frac{\delta(\varphi,x_0,t\e_k)}{|t|^{1+2s}}\leq
\frac{\min\{2[\varphi]_{\Lip(\R^n)}|t|,\SC(\varphi)|t|^2\}}{|t|^{n+2s}}\in L^1(\R),
\]
uniformly in $k\geq1$. 
Thus, we can apply the dominated convergence theorem to \eqref{eq:almostthere} to get
\[
-(-\Delta)_{\e_0}^s\varphi(x_0)=-\lim_{k\to\infty}(-\Delta)_{\e_k}^s\varphi(x_0)\leq0,
\]
as desired.
\end{proof}

\begin{lem}
\label{lem: u0 supersolution}
Let $u_0$ be as in \eqref{uzero}. Then,
\[
\D u_0\leq u_0-\phi\quad\hbox{in}~\{u_0<\psi\}.
\]
\end{lem}

\begin{proof}
Let $x_0 \in \{ u_0 < \psi \}$ be a point at which
$u_0$ can be touched from below by a $C^2$ function in a neighborhood $\Nhd\subset\subset\{u_0<\psi\}$.
As $u_\vep$ decreases locally uniformly to $u_0$, we can find a sequence of points $x_{\vep}\to x_0$ and $C^2$ functions $P_{\vep}$ that touch $u_{\vep}$ from below at $x_{\vep}$ in a common neighborhood $\Nhd'\subset\subset\Nhd$.
By Lemma \ref{lem:uzero separated}, we have $u_\vep \geq u_0>\phi$.
Then, by Theorem \ref{thm:our-ellipticity}, there exists $\lambda > 0$ such that
\[
\D u_0(x_0)=\D^\lambda u_0(x_0)\quad\text{and}\quad\D^{\vep}u_{\vep}(x_{\vep})
= \D^\lambda u_{\vep}(x_{\vep})
\]
for every $\vep$ such that $\vep<\lambda$. 
Now, since $u_\vep$ is a supersolution in $\{u_\vep<\psi\}$ (see Theorem~\ref{thm: eps-obstacle problem}) that can be touched from below by a $C^2$ function at $x_\vep$ in $\Nhd'$, we can apply Lemma \ref{lem:touchandcompute}. Consequently,
\begin{equation}
\label{alpha}
\D^\lambda u_{\vep}(x_{\vep}) = \D^{\vep} u_{\vep}(x_{\vep})
\leq u_{\vep}(x_{\vep})-\phi(x_{\vep}).
\end{equation}
By Lemma~\ref{lem:gammaconvergence},
\[
\lim_{\vep\to0}\D^\lambda u_{\vep}(x_{\vep})=\D^\lambda u_0(x_0).
\]
Thus, by letting $\vep\to0$ in \eqref{alpha},
\[
\D u_0(x_0) = \D^\lambda u_0(x_0) \leq u_0(x_0) -\phi(x_0).
\]
\end{proof}

With the following result we are able to conclude the proof of Theorem \ref{thm: solution of obstacle problem}.

\begin{lem}\label{lem:u equals uzero}
Let $u$ and $u_0$ be as in \eqref{eq: u} and \eqref{uzero}, respectively.
Then,
\[ 
u=u_0.
\]
\end{lem}

\begin{proof}
Let us show that $u_0\in\mathpzc{F}$. 
Let $A\in\mathpzc{M}$. 
If $\vep$ is smaller than the minimum eigenvalue of $A$, then $A\in\mathpzc{M}_\vep$, and, as such, we have $L_A^s u_\vep\geq\De u_\vep \geq u_\vep - \phi$ in $\R^n$.
Therefore, by Lemma \ref{lem:gammaconvergence}, we find that $L_A^s u_0 \geq u_0 - \phi$ in $\R^n$. 
As $A\in\mathpzc{M}$ was arbitrary, it follows that 
\[
\D u_0 \geq u_0 - \phi \quad\hbox{in}~\R^n.
\]
Therefore, $u_0\in\mathpzc{F}$ and
$u_0\leq u$.
For the opposite inequality, observe that
\[
\De u \geq \D u \geq u - \phi \quad\hbox{in}~\R^n.
\]
Whence, $u \in \mathpzc{F}_\vep$ for all $\vep > 0$, and
by the maximality of $u_\vep$ in $\mathpzc{F}_\vep$, $u \leq u_\vep$ for all $\vep > 0$.
Thus, from the definition of $u_0$, we determine that
$u \leq u_0$.
\end{proof}

\begin{proof}[Proof of Theorem \ref{thm: solution of obstacle problem}]
The proof follows from Lemmas~\ref{lem: u subsolution}--\ref{lem:classical soln} and Lemmas~\ref{lem:uzero separated}--\ref{lem:u equals uzero}.
\end{proof}

\section{Proof of Theorem \ref{thm:regularity}}
\label{sec: higher reg}

We first prove Theorem~\ref{thm:regularity}(1), 
that is, the local H\"older continuity of $\nabla u$ in the noncoincidence set.

\begin{proof}[Proof of Theorem \ref{thm:regularity}(1)]
Let $\mathcal{O}$ and $\mathcal{O}_\delta$ be as in the statement.
Since $u>\phi$ in $\R^n$, by Theorem~\ref{thm:our-ellipticity}, there exists
$\lambda=\lambda(n,s,\inf_{\mathcal{O}_\delta}(u-\phi),M_1,M_2)>0$ such that
\[
\D u(x)=\D^\lambda u(x)\quad\hbox{for all}~x\in\mathcal{O}_\delta.
\]
For any $A\in\mathpzc{M}_\lambda$, let
\[
b_A(x)=L_A^s\phi(x)-(u-\phi)(x).
\]
Since $u$ and $\phi$ are Lipschitz, $\phi \in C^{2,\sigma}(\R^n)$ and satisfies \eqref{M0}, and $\lim_{|x|\to \infty} (u-\phi)(x) = 0$,
we deduce that
\begin{equation}
\label{eq:cA}
\sup_{A\in\mathpzc{M}_\lambda}\big{\{}\| b_A \|_{L^\infty(\mathcal{O}_\delta)} + [b_A]_{\Lip(\mathcal{O}_\delta)}\big{\}}\leq C_0,
\end{equation}
where $C_0 = C_0(n,s,\lambda,M_0,M_1,\SC(\phi)) > 0$.
Let $w=u-\phi$.  We have
\[
L_A^sw+b_A(x)=L_A^su-(u-\phi)\quad
\hbox{in}~\mathcal{O}_\delta.
\]
By taking the infimum over all $A\in\mathpzc{M}_\lambda$ above, we see that $w$ solves
\[
\begin{cases}
\displaystyle\inf_{A\in\mathpzc{M}_\lambda}\big\{L_A^sw+b_A(x)\big\}=0&\hbox{in}~\mathcal{O}_\delta\\
w=u-\phi\in L^\infty(\R^n),
\end{cases}
\]
with $b_A$ satisfying the uniform estimate \eqref{eq:cA}.
From Theorem 1.3(b) in \cite{Serra}, the conclusion follows.
\end{proof}

Next we prove Theorem~\ref{thm:regularity}(2), which establishes that $\nabla u$ is H\"older continuous across the free boundary.
Recall that the contact set $\{u=\psi\}$ is compact, see Lemma~\ref{lem:compact}.
Let $\mathcal{B}$ be as in the statement.
Set $C\mathcal{B}=\{Cx:x\in\mathcal{B}\}$ with $C > 0$.
By Theorem~\ref{thm:our-ellipticity}, there exists
\begin{equation}
\label{lambda}
\lambda=\lambda(n,s,\inf_{4\mathcal{B}}(u-\phi),M_1,M_2)>0
\end{equation}
such that
\begin{equation}\label{eq:3B}
\D u=\D^\lambda u\quad\hbox{in}~4\mathcal{B}.
\end{equation}
Define
\begin{equation}\label{eq:cAx}
c_A(x)=(u-\phi)(x)-L_A^s\psi(x).
\end{equation}
Since
$\sup_{A\in\mathpzc{M}_\lambda}[L_A^s\psi]_{\Lip(\R^n)}<\infty$
and $u,\phi\in\Lip(\R^n)$, up to dividing by a constant depending on $\lambda$, 
we can assume that
\begin{equation}\label{lipca}
\sup_{A\in\mathpzc{M}_\lambda}[c_A]_{\Lip(\R^n)}=1.
\end{equation}
We subtract the obstacle and let $v$ be as in \eqref{eq:psi minus u}.
For any $A\in\mathpzc{M}$, we have
\[
L_A^sv+c_A(x)=-L_A^su+(u-\phi),
\]
so that
\[
\sup_{A\in\mathpzc{M}_\lambda}\big{\{}L_A^sv+c_A(x)\big{\}}=-\D^\lambda u+(u-\phi).
\]
Therefore, from \eqref{eq:3B} and up to dividing $v$ by a normalizing constant depending on $\lambda$, we get
\begin{equation}\label{eq:obstacle v}
\begin{cases}
v\geq0&\hbox{in}~\R^n\\
D^2v(x)\geq -\Id&\hbox{for a.e.}~x\in \R^n \\
\sup_{A\in\mathpzc{M}_\lambda}\big{\{}L_A^sv(x)+c_A(x)\big{\}}=0&\hbox{in}~\{v>0\}\cap 4\mathcal{B} \\
\sup_{A\in\mathpzc{M}_\lambda}\big{\{}L_A^sv(x)+c_A(x)\big{\}}\leq 0&\hbox{in}~\R^n \\
|\nabla v(x)|\leq 1&\hbox{for a.e.}~x\in\R^n.
\end{cases}
\end{equation}
Finally, consider the extremal Pucci operators
\[
M^+_\lambda w(x)=\sup_{A\in\mathpzc{M}_\lambda}L_A^sw(x)
\quad\hbox{and}\quad
M^-_\lambda w(x)=\inf_{A\in\mathpzc{M}_\lambda}L_A^sw(x).
\]
To prove Theorem \ref{thm:regularity}(2),
we need the following rescaled
version of a regularity result from \cite{Caffarelli-Ros-Oton-Serra}.

\begin{prop}
\label{Proposition 2.4}
Let $\alpha\in(0,s)$, $1+s+\alpha<2$,  $K>0$, and $R\geq1$. If $w$ satisfies
\[
\begin{cases}
w\geq0&\hbox{in}~\R^n\\
D^2w(x)\geq -K\Id&\hbox{for a.e.}~x\in B_{2R} \\
M^+_\lambda(w-w(\cdot-h))\geq -K|h|&\hbox{in}~\{w>0\}\cap B_R \\
|\nabla w(x)|\leq K(1+|x|^{s+\alpha})&\hbox{for a.e.}~x\in\R^n,
\end{cases}
\]
then there exist $0<\tau<1$ and $C>0$, depending on $\alpha$ and
$\lambda$, such that
\[
\|w\|_{L^\infty(B_{R/2})}+R\|\nabla w\|_{L^\infty(B_{R/2})}+R^{1+\tau}[\nabla w]_{C^\tau(B_{R/2})}
\leq CKR^2.
\]
\end{prop}

\begin{proof}[Proof of Theorem \ref{thm:regularity}(2)]
Fix $\mathcal{B}$ as in the statement. Let $\lambda$ be as in \eqref{lambda},
and let $v$ be as in \eqref{eq:psi minus u}. Observe that, by \eqref{eq:obstacle v} and \eqref{lipca},
\begin{align*}
M^+_\lambda(v-v(\cdot-h))(x) &\geq (L_A^sv+c_A)(x)-(L_A^sv+c_A)(x-h)-(c_A(x)-c_A(x-h)) \\
&\geq(L_A^sv+c_A)(x)-|h|,
\end{align*}
which gives
\[
M^+_\lambda(v-v(\cdot-h)) \geq\sup_{A\in\mathpzc{M}_\lambda}\big{\{}L_A^sv+c_A(x)\big{\}}-|h|=-|h| \quad\hbox{in}~\{v>0\}\cap 2\mathcal{B}.
\]
With this and \eqref{eq:obstacle v}, we can apply Proposition \ref{Proposition 2.4} and conclude that $v\in C^{1,\tau}(\mathcal{B})$,
with the corresponding estimate.
\end{proof}

\section{Proof of Theorem \ref{thm:blow up}}
\label{Section:blow up}

In order to prove Theorem \ref{thm:blow up}, we consider $v=\psi-u$ as in \eqref{eq:psi minus u}. 
We showed, in section \ref{sec: higher reg}, that $v$
satisfies the locally uniformly elliptic obstacle problem \eqref{eq:obstacle v}
with ellipticity constants $\lambda >0$ (as defined in \eqref{lambda}) and $1/\lambda^{(n-1)(n+2s)}$.
Before proceeding with the proof, we define regular free boundary points,
the constant $\bar{\alpha}>0$, and the rescalings we will use to determine the blow up sequence,
by following \cite{Caffarelli-Ros-Oton-Serra}.

\begin{defn}
\label{defn:regular}
Let $\nu:(0,\infty)\to(0,\infty)$ be a nonincreasing function with
\[
\lim_{r\searrow0}\nu(r)=\infty.
\]
We say that a free boundary point $x_0\in\partial\{v>0\}$ is \textit{regular with modulus $\nu$} if
\[
\sup_{\rho\geq r}\frac{\sup_{B_{\rho}(x_0)}v}{\rho^{1+s+\alpha}}\geq\nu(r)
\]
for some $\alpha\in(0,s)$ such that
\[
1+s+\alpha<2.
\]
\end{defn}

\begin{defn}
Let $\bar{\alpha}=\bar{\alpha}(n,s,\lambda)>0$ be the minimum of the following three constants:
\begin{itemize}
\item[-] The $\alpha>0$ of the interior $C^\alpha$ estimate given by \cite[Theorem~11.1]{Caffarelli-Silvestre-CPAM};
\item[-] The $\alpha>0$ of the boundary $C^\alpha$ estimate for $u/d^s$ given by \cite[Proposition 1.1]{Ros-Oton-Serra-Duke};
\item[-] The $\alpha>0$ of the interior $C^{2s+\alpha}$ estimate for convex equations given by \cite[Theorem 1.1]{Caffarelli-Silvestre-ARMA}
and \cite[Theorem 1.1]{Serra}.
\end{itemize}
\end{defn}

Without loss of generality and for the rest of this section, we assume that $x_0=0$ is a regular free boundary point
with modulus $\nu$. In the case $1+s+\alpha\geq 2s+\bar{\alpha}$, we further assume that
\[
\liminf_{\rho \searrow 0} \frac{|\{v =0\} \cap B_\rho|}{|B_\rho|} > 0.
\]
In particular, there exists $c_0>0$ such that 
\begin{equation}\label{eq:positive density}
\frac{|\{v =0\} \cap B_\rho|}{|B_\rho|} \geq c_0> 0\quad\hbox{for all}~\rho~\hbox{sufficiently small}.
\end{equation}

For $r>0$, we define the rescalings
\begin{equation}
\label{eq:vr}
v_r(x)=\frac{v(rx)}{r^{1+s+\alpha}\theta(r)}\quad\hbox{for}~x\in\R^n
\end{equation}
where
\[
\theta(r)=\sup_{\rho\geq r}\frac{\|\nabla v\|_{L^\infty(B_\rho)}}{\rho^{s+\alpha}}.
\]
Then, $\theta$ is nonincreasing and $\theta(r)\geq\nu(r)$ for all $r>0$, see \cite[Lemma~5.4]{Caffarelli-Ros-Oton-Serra}.

\begin{proof}[Proof of Theorem \ref{thm:blow up}]
We first prove that for any $R>0$,
$\|v_r\|_{C^{1,\tau}(B_R)}$ is uniformly bounded for all $r>0$ sufficiently small,
where $\tau\in(0,1)$ is as in Proposition \ref{Proposition 2.4}.
Observe that
\begin{equation}
\label{v1}
v_r\geq0\quad\hbox{in}~\R^n,
\end{equation}
and, since $1+s+\alpha<2$ and $\theta(r)\geq\nu(r)$, for all $r<1$,
\begin{equation}
\label{v2}
D^2v_r(x)=\frac{r^2}{r^{1+s+\alpha}\theta(r)}D^2v(rx)\geq
-\frac{1}{\nu(r)}\Id\quad\hbox{for a.e.}~x\in\R^n.
\end{equation}
Let $c_A(x)$ be as in \eqref{eq:cAx} and define
\[
c_{A,r}(x)=\frac{c_A(rx)}{r^{1-s+\alpha}\theta(r)}.
\]
Then, for every $r<1$, since $\alpha\in(0,s)$ and \eqref{lipca} holds,
\[
[c_{A,r}]_{\Lip(\R^n)}\leq \frac{r^{s-\alpha}}{\theta(r)}[c_A]_{\Lip(\R^n)}\leq \frac{1}{\nu(r)}.
\]
Hence, following the proof of Theorem \ref{thm:regularity}(2) in section \ref{sec: higher reg}, we see that
\begin{equation}
\label{v3}
M^+_\lambda(v_r-v_r(\cdot-h))(x)\geq -\frac{|h|}{\nu(r)}\quad\hbox{for all}~x\in\{v_r>0\}\cap B_{4R}
\end{equation}
and for all $r<1$ sufficiently small. Finally,
\begin{equation}
\label{v4}
\|\nabla v_r\|_{L^\infty(B_R)}=\frac{R^{s+\alpha}\|\nabla v\|_{L^\infty(B_{Rr})}}{(Rr)^{s+\alpha}\theta(r)}\leq R^{s+\alpha},
\end{equation}
which implies that
\[
|\nabla v_r(x)| \leq 2(1+|x|^{s+\alpha})\quad\hbox{for all}~x\in\R^n.
\]
Thus, by Proposition \ref{Proposition 2.4}, $\|v_r\|_{C^{1,\tau}(B_R)}$ is uniformly bounded for all $r>0$ sufficiently small.

Next, for any $k\geq1$, we choose
\begin{equation}
\label{eqn:rk}
r_k\geq\frac{1}{k}
\end{equation}
such that
\begin{equation}
\label{eq:rk}
\frac{\|\nabla v\|_{L^\infty(B_{r_k})}}{r_k^{s+\alpha}}\geq\frac{1}{2}\theta(1/k)\geq \frac{1}{2}\theta(r_k).
\end{equation}
Since $\|\nabla v\|_{L^\infty(\R^n)}\leq 1$ and $\theta(1/k)\geq\nu(1/k)\to\infty$ as $k\to\infty$,
we have that $r_k\to0$ as $k\to\infty$. In addition, from \eqref{eq:rk} it follows that
\begin{equation}
\label{v5}
\|\nabla v_{r_k}\|_{L^\infty(B_1)}\geq\frac{1}{2}.
\end{equation}
Moreover, in the case $1+s+\alpha\geq 2s+\bar{\alpha}$, by \eqref{eq:positive density},
\begin{equation}
\label{v6}
\frac{|\{v_{r_k} =0\} \cap B_\rho|}{|B_\rho|} = \frac{|\{v =0\} \cap B_{r_k\rho}|}{|B_{r_k\rho}|}\geq c_0 > 0\quad\hbox{for all}~\rho~\hbox{sufficiently small}.
\end{equation}
By Arzel\`a--Ascoli's theorem and a standard diagonal argument, there exists $v_0\in C^1_{\mathrm{loc}}(\R^n)$
such that
\[
v_{r_k}\to v_0\quad\hbox{in}~C^1_{\mathrm{loc}}(\R^n).
\]
Additionally, by \eqref{v1} and \eqref{v2},
\[
\begin{cases}
v_0\geq0&\hbox{in}~\R^n\\
D^2v_0(x)\geq 0&\hbox{for a.e.}~x\in\R^n,
\end{cases}
\]
and, by \eqref{v4} and \eqref{v5},
\[
\frac{1}{2}\leq\|\nabla v_0\|_{L^\infty(B_R)}\leq R^{s+\alpha}\quad\hbox{for all}~R\geq1.
\]
Also, in the case $1+s+\alpha\geq 2s+\bar{\alpha}$, from \eqref{v6}, we find that
\[
\liminf_{\rho\searrow0}\frac{|\{v_0 =0\} \cap B_\rho|}{|B_\rho|}\geq c_0>0.
\]
Next, as $R>0$ in \eqref{v3} was arbitrary, by passing to the limit as $|h|\to0$ and $r=r_k\to0$, we get
\[
M^+_\lambda(\partial_\e v_0)\geq 0\quad\hbox{in}~\{v_0>0\}
\]
and for all $\e\in\mathbb{S}^{n-1}$.
Furthermore, by arguing as we did to obtain \eqref{v3}, for any $R>0$
and any nonnegative probability measure $\mu$ with compact support, we find that
\[
M^+_\lambda\bigg(v_r-\int v_r(\cdot-h)\,d\mu(h)\bigg)\geq-\frac{|h|}{\nu(r)}\quad\hbox{in}~\{v_r>0\}\cap B_{R}
\]
provided $r$ is sufficiently small. As a consequence,
\[
M^+_\lambda\bigg(v_0-\int v_0(\cdot-h)\,d\mu(h)\bigg)\geq0\quad\hbox{in}~\{v_0>0\}.
\]
Therefore, applying the classification results in \cite[Theorems~7.1,~7.2]{Caffarelli-Ros-Oton-Serra} to $v_0$, we finally
obtain
\[
v_0(x)=K_0(\e_0\cdot x)^{1+s}_+,
\] 
for some $1/4\leq K_0\leq 1$ and $\e_0\in\S^{n-1}$, as desired.
\end{proof}

\section{Proof of Theorem \ref{thm:2}}
\label{Section:fb}

As in section~\ref{Section:blow up}, we assume, without loss of generality, that $x_0 = 0$ is a regular free boundary point with modulus $\nu$.

\begin{proof}[Proof of Theorem \ref{thm:2}]
Consider $v_r$ as in \eqref{eq:vr}, for $r = r_k$ given by \eqref{eqn:rk}.
From Theorem \ref{thm:blow up} and its proof, we have that given any 
$\delta_0>0$ and $R_0\geq1$, there exists $r_0=r_0(\delta_0,R_0,\alpha,\nu,\lambda)\in(0,1)$ such that
for all $0<r_k<r_0$,
\[
\begin{cases}
M^+_\lambda(\partial_\e v_{r_k})\geq -\delta_0&\hbox{in}~\{v_{r_k}>0\}\cap B_{R_0}\\
M^-_\lambda(\partial_\e v_{r_k})\leq\delta_0&\hbox{in}~\{v_{r_k}>0\}\cap B_{R_0},
\end{cases}
\]
for all $\e\in\mathbb{S}^{n-1}$, and
\[
|v_{r_k}(x)-K_0(\e_0\cdot x)^{1+s}_+|+|\nabla v_{r_k}(x)-(1+s)K_0(\e_0\cdot x)^{s}_+\e_0|<\delta_0\quad\hbox{for every}~x\in B_{R_0}.
\]
In addition, for every $R\geq1$, we have
\[
\|\nabla v_{r_k}\|_{L^\infty(B_{R})}\leq R^{s+\alpha},
\]
for all $k\geq1$, see \eqref{v4}. With these estimates in hand,
we can argue exactly as in sections 8.2 and 8.3 of \cite{Caffarelli-Ros-Oton-Serra}
and deduce that Theorem \ref{thm:2} holds.
\end{proof}

\medskip

\noindent\textbf{Acknowledgements.} We are grateful to Luis A.~Caffarelli for posing this problem to us.



\end{document}